\documentclass[reqno,a4paper,twoside]{amsart}

\usepackage{geometry}
\usepackage{enumitem}

\setenumerate{label=\textnormal{(\arabic*)}}

\allowdisplaybreaks[4]

\usepackage{amsmath,amssymb,dsfont,verbatim,bm,geometry,mathrsfs}

\usepackage{mathtools}
\usepackage[raggedright]{titlesec}

\usepackage[utf8]{inputenc}
\usepackage[T1]{fontenc}

\usepackage{hyperref}
\hypersetup{
	colorlinks=true,
	linkcolor=blue,
	filecolor=magenta,
	urlcolor=cyan,
	unicode=true,
	pdfpagemode=FullScreen,
}

\titleformat{\chapter}[display]
{\normalfont\huge\bfseries}{\chaptertitlename\\thechapter}{20pt}{\Huge}
\titleformat{\section}
{\normalfont\Large\bfseries\center}{\thesection}{1em}{}
\titleformat{\subsection}
{\normalfont\large\bfseries}{\thesubsection}{1em}{}
\titleformat{\subsubsection}
{\normalfont\normalsize\bfseries}{\thesubsubsection}{1em}{}
\titleformat{\paragraph}[runin]
{\normalfont\normalsize\bfseries}{\theparagraph}{1em}{}
\titleformat{\subparagraph}[runin]
{\normalfont\normalsize\bfseries}{\thesubparagraph}{1em}{}

\titlespacing*{\chapter} {0pt}{50pt}{40pt}
\titlespacing*{\section} {0pt}{3.5ex plus 1ex minus .2ex}{2.3ex plus .2ex}
\titlespacing*{\subsection} {0pt}{3.25ex plus 1ex minus .2ex}{1.5ex plus .2ex}
\titlespacing*{\subsubsection}{0pt}{3.25ex plus 1ex minus .2ex}{1.5ex plus .2ex}
\titlespacing*{\paragraph} {0pt}{3.25ex plus 1ex minus .2ex}{1em}
\titlespacing*{\subparagraph} {\parindent}{3.25ex plus 1ex minus .2ex}{1em}

\usepackage{tikz}
\usetikzlibrary{matrix, shapes}

\usepackage{float, subfig}
\usepackage{amsrefs}

\usepackage{titletoc}

\contentsmargin{2.55em}
\dottedcontents{section}[1.5em]{}{2em}{1pc}
\dottedcontents{subsection}[4.35em]{}{2.8em}{1pc}
\dottedcontents{subsubsection}[7.6em]{}{3.2em}{1pc}
\dottedcontents{paragraph}[10.3em]{}{3.2em}{1pc}

\makeindex

\newtheorem{theorem}{Theorem}[section]
\newtheorem{lemma}[theorem]{Lemma}
\newtheorem{proposition}[theorem]{Proposition}
\newtheorem{corollary}[theorem]{Corollary}

\theoremstyle{definition}
\newtheorem{definition}[theorem]{Definition}

\newtheorem{example}[theorem]{Example}

\theoremstyle{remark}
\newtheorem{remark}[theorem]{Remark}

\usepackage{tikz}
\usetikzlibrary{matrix,shapes}
\usetikzlibrary{cd}

\DeclareMathOperator{\End}{End}

\DeclareMathOperator{\cop}{cop}

\DeclareMathOperator{\ide}{id}

\DeclareMathOperator{\Soc}{Soc}

\DeclareMathOperator{\Z}{Z}

\DeclareMathOperator{\op}{op}

\newcommand{\xcirc}{\hspace{0.9pt}}

\newcommand{\ot}{\otimes}

\newcommand{\hs}{\hspace{-0.7pt}}

\newcommand{\dpu}{\mathbin{:}}

\numberwithin{equation}{section}

\DeclareMathAlphabet{\mathpzc}{OT1}{pzc}{m}{it}

\usepackage{mathtools}
\newtagform{brackets}{[}{]}
\usetagform{brackets}

\begin{document}
	
\title{Hopf $\mathbf{Q}$-braces structures on rank one pointed Hopf algebras}
	
	\author{Jorge A. Guccione}
	\address{Departamento de Matem\'atica\\ Facultad de Ciencias Exactas y Naturales-UBA, Pabell\'on~1-Ciudad Universitaria\\ Intendente Guiraldes 2160 (C1428EGA) Buenos Aires, Argentina.}
	\address{Instituto de Investigaciones Matem\'aticas ``Luis A. Santal\'o''\\ Pabell\'on~1-Ciudad Universitaria\\ Intendente Guiraldes 2160 (C1428EGA) Buenos Aires, Argentina.}
	\email{vander@dm.uba.ar}
	
	\author{Juan J. Guccione}
	\address{Departamento de Matem\'atica\\ Facultad de Ciencias Exactas y Naturales-UBA\\ Pabell\'on~1-Ciudad Universitaria\\ Intendente Guiraldes 2160 (C1428EGA) Buenos Aires, Argentina.}
	\address{Instituto Argentino de Matem\'atica-CONICET\\ Saavedra 15 3er piso\\ (\!C1083ACA\!) Buenos Aires, Argentina.}
	\email{jjgucci@dm.uba.ar}
	
	\thanks{Jorge A. Guccione and Juan J. Guccione were supported by CONICET PIP 2021-2023 GI,11220200100423CO and CONCYTEC-FONDECYT within the framework of the contest ``Proyectos de Investigaci\'on B\'asica 2020-01'' [contract number 120-2020-FONDECYT]}
	
	\author{Christian Valqui}
	\address{Pontificia Universidad Cat\'olica del Per\'u, Secci\'on Matem\'aticas, PUCP, Av. Universitaria 1801, San Miguel, Lima 32, Per\'u.}
	
	\address{Instituto de Matem\'atica y Ciencias Afines (IMCA) Calle Los Bi\'ologos 245. Urb San C\'esar. La Molina, Lima 12, Per\'u.}
	\email{cvalqui@pucp.edu.pe}
	
\thanks{Christian Valqui was supported by PUCP-CAP 2023-PI0991}
	
	
	
\begin{abstract}
In this paper we determine all the Hopf $q$-brace structures on rank one pointed Hopf algebras and compute the socle of each one of them. We also identify which among them are Hopf skew-braces. Then we determine when two Hopf $q$-brace structures on rank one pointed Hopf algebras are isomorphic, and, finally, we compute all the weak braiding operators on these Hopf algebras.
\end{abstract}	
	
\maketitle
	
\tableofcontents
	
\section*{Introduction}
Let $V$ be a vector space over a field $\mathds{k}$ and let $r\colon V \ot V \to V \ot V$ be a linear map. We say that $r$ satisfies the Yang-Baxter equation on $V$ if
$$
r_{12} \xcirc r_{13}\xcirc r_{23} = r_{23} \xcirc r_{13} \xcirc r_{12}\quad \text{in $\End_{\mathds{k}}(V \ot V \ot V)$,}
$$
where $r_{ij}$ means $r$ acting in the $i$-th and $j$-th components. This occurs if and only if $s\coloneqq \tau\xcirc r$,  where $\tau$ denotes the flip, satisfies the braid equation $s_{12} \xcirc s_{23} \xcirc s_{12} = s_{23} \xcirc s_{12} \xcirc s_{23}$. In~\cite{D}, Drinfeld posed the problem of studying this equation from the set-theoretical point of view. A set theoretic solution of the braid equation is a pair $(Y,s)$, where $s\colon Y\times Y\to Y\times Y$ is a map satisfying this equation. This structure was investigated by many mathematicians due to their connection with semigroup and group theory, knot theory, Hopf algebras, Artin-Schelter regular rings, etcetera. See for instance \cites{AGV, CAV, DG, De1, DDM, Ch, ESS, GI1, GI2, GI3, GI4, GIVB, JO, LYZ}.

Two important notions in this theory are those of $q$-braces (\cite{R2}*{Definition~7}) and skew-braces (\cite{GV}*{Definition~1.1}). A $q$-brace is a group $G$ endowed with two binary operations $(g,h)\mapsto g\cdot h$ and $(g,h)\mapsto g\dpu h$ satisfying suitable conditions. In~\cite{R2}*{Corollary~2 of Proposition~6} it was proved that a skew-brace is a $q$-brace if and only if 
\begin{equation}\label{skew-brace Rump}
(h\dpu g)h=(g\cdot h)g,\quad \text{for all $g,h\in G$.} 
\end{equation}
In the paper \cite{GGV}, we introduced  the notions of Hopf $q$-brace and Hopf skew-brace and we begin the study of their properties. A Hopf $q$-brace $(H,\cdot, \dpu)$ is a Hopf algebra $H$ endowed with binary operations $h\ot l\mapsto h\cdot l$ and $h\ot l\mapsto h\dpu l$ satisfying several conditions (see Definition~\ref{q-brace}). A Hopf skew-brace is a Hopf $q$-brace such that
\begin{equation}\label{skew-brace}
(h\dpu l_{(2)})l_{(1)}=(l\cdot h_{(1)})h_{(2)}, 
\end{equation}
for all $h,l\in H$, where we are using the Sweedler notation (note that our choice of laterality in~\eqref{skew-brace}, for the operations $\cdot$ and $\dpu$, is opposite to the one of~\eqref{skew-brace Rump}). For the relation between Hopf $q$-brace and solutions of the Yang-Baxter equations see \cite{GGV}*{Corollary~4.11, Theorem~5.15 and Proposition~6.4}.

The aim of this paper is to classify all the Hopf $q$-braces whose underlying Hopf algebra is a rank one pointed Hopf algebra over an algebraically closed field of characteristic zero $\mathds{k}$ (see~\cite{KR}) and to determine which ones are Hopf skew-braces. The paper is organized as follows: in Section~\ref{preliminares} we make a very brief review of the notions of Hopf skew-braces and rank one pointed Hopf algebras. In Section~\ref{section 2} we begin the study of the Hopf $q$-brace structures on rank one pointed Hopf algebras. The main result is Proposition~\ref{resumen}. In Section~\ref{Section 3}, in Theorems~\ref{clasificacion 1} and~\ref{n=2}, we determine all the Hopf $q$-brace structures on rank one pointed Hopf algebras. In Section~\ref{Section 3}, we also determine which ones of these structures are Hopf skew-braces. We finish this section determining when two Hopf $q$-braces on rank one pointed Hopf algebras are isomorphic. Then, in Section~\ref{section 4}, we compute the socle of the Hopf $q$-braces obtained in the previous section. Finally, in Section~\ref{section 5}, we use the results from Section~\ref{Section 3} to compute the weak braiding operators on rank one pointed Hopf algebras.

\smallskip

\noindent {\bf Precedence of operations}\enspace The operations precedence in this paper is as follows: the operators with the highest precedence are the binary operations $a^b$ and ${}^ba$, followed by the multiplication $ab$. Next come the binary operations $\cdot$ and $\dpu$, and finally, the operations $+$ and $-$.

\section{Preliminaries}\label{preliminares}                  

\subsection[Hopf \texorpdfstring{$q$}{q}-braces]{Hopf $\mathbf{q}$-braces}
Let $X$ be a coalgebra and let $p,d\colon X^2\to X$ be maps such that $\epsilon\xcirc p=\epsilon\xcirc d=\epsilon\ot \epsilon$. For each $x,y\in X$, we set $x\cdot y\coloneqq p(x\ot y)$ and $x\dpu y\coloneqq d(x\ot y)$. Let $\overline{G}_{\mathcal{X}}\in \End_{\mathds{k}}(X^2)$ be the map defined by $\overline{G}_{\mathcal{X}}(x\ot y)\coloneqq x\cdot {y_{(1)}}\ot y_{(2)}$.

\begin{definition}\label{q-magma coalgebra} We say that $\mathcal{X}=(X,\cdot,\dpu)$ is a  {\em $q$-magma coalgebra} if the map $h\colon X\ot X^{\cop}\to X^{\cop}\ot X$, defined by $h(x\ot y)\coloneqq y_{(1)}\dpu x_{(2)}\ot x_{(1)}\cdot y_{(2)}$, is a coalgebra morphism. If necessary, we will write $h_{\mathcal{X}}$ instead of $h$. 
\end{definition}

\begin{remark}\label{d y p son morfismos de coalgears}  $\mathcal{X}$ is a $q$-magma coalgebra if and only if $p,d\colon X\ot X^{\cop}\to X$ are coalgebra maps and
\begin{equation}\label{intercambio . :}
v_{(1)}\dpu u_{(2)}\ot u_{(1)}\cdot v_{(2)}=v_{(2)}\dpu u_{(1)}\ot u_{(2)}\cdot v_{(1)}\qquad \text{for all $u,v\in X$.}
\end{equation}
\end{remark}

\begin{remark}\label{opuesto de q magma coalgebra} If $\mathcal{X}=(X,\cdot,\dpu)$ is a $q$-magma coalgebra, then $\mathcal{X}^{\op}\coloneqq (X^{\cop},\dpu, \cdot)$ is also. We call~$\mathcal{X}^{\op}$~the~{\em~op\-po\-site} $q$-magma coalgebra of $\mathcal{X}$. Equality~\eqref{intercambio . :} says that $h_{\mathcal{X}^{\op}}= \tau\xcirc h_{\mathcal{X}}\xcirc \tau$.
\end{remark}

\begin{definition}\label{q-magma coalgebra no degenerada a izquierda} A $q$-magma coalgebra $\mathcal{X}$ is {\em left regular} if the map $\overline{G}_{\mathcal{X}}$ is invertible. If~$\mathcal{X}^{\op}$ is left regular, then we call $\mathcal{X}$ {\em right regular}. Finally, we say that $\mathcal{X}$ is {\em regular} if it is left regular and right regular.
\end{definition}

\begin{definition}\label{def: q cycle coalgebra} A {\em $q$-cycle coalgebra} is a left regular $q$-magma coalgebra $\mathcal{X}=(X,\cdot,\dpu)$ such that:
\begin{align}		
& (u \cdot v_{(1)})\cdot (w\dpu v_{(2)})=(u\cdot w_{(2)})\cdot (v\cdot w_{(1)}),\label{1 de def: q cycle coalgebra}\\
&(u \cdot v_{(1)})\dpu (w\cdot v_{(2)})=(u\dpu w_{(2)})\cdot (v\dpu w_{(1)}),\label{2 de def: q cycle coalgebra}\\
& (u \dpu  v_{(1)})\dpu (w\dpu v_{(2)})=(u\dpu w_{(2)})\dpu (v\cdot w_{(1)}),\label{3 de def: q cycle coalgebra}
\end{align}
for all $u,v,w\in X$.
\end{definition}

\begin{definition}\label{q-brace} Let $H$ be a Hopf algebra and let $\mathcal{H}=(H,\cdot,\dpu)$ be a regular $q$-cycle coalgebra. We~say~that~$\mathcal{H}$ is a {\em Hopf $q$-brace} if $(H,\cdot)$ and $(H,\dpu)$ are right $H^{\op}$-modules and the following equalities hold:
\begin{equation}\label{condicion q-braza}
hk\cdot l = (h\cdot (l_{(1)}\dpu k_{(2)}))(k_{(1)}\cdot l_{(2)})\quad\text{and}\quad hk\dpu l = (h\dpu (l_{(1)}\cdot k_{(2)}))(k_{(1)}\dpu l_{(2)}),
\end{equation}
for all $h,k,l\in H$. Let $\mathcal{K}=(K,\cdot,\dpu)$ be another Hopf $q$-brace. A {\em morphism} $f\colon \mathcal{H}\to \mathcal{K}$ is a Hopf algebra morphism $f\colon H\to K$ such that $f(h\cdot h')=f(h)\cdot f(h')$ and $f(h\dpu h')=f(h)\dpu f(h')$, for all $h,h'\in H$.
\end{definition}

\subsection{Rank one pointed Hopf algebras}
Let $\mathds{k}$ be an algebraically closed field of characteristic zero. A tuple $\mathcal{D}=(G,\chi,a,\alpha)$, consisting of a finite group $G$, a character $\chi\colon G\to \mathds{k}^{\times}$, an element $a\in \Z(G)$ and an element $\alpha\in \mathds{k}$, is a {\em datum} if $\alpha(a^n-1)=0$ or $\chi^n=1$, where $n>1$ is the order of $\chi(a)$.

Given a datum $\mathcal{D}$, {\em the rank one pointed Hopf algebra associated with $\mathcal{D}$} is the Hopf algebra $H_{\mathcal{D}}$ generated as algebra by $G$ and $x$, subject to the group relations for $G$,
$$
x^n=\alpha (a^n-1)\quad\text{and}\quad xg=\chi(g)gx\quad\text{for all $g\in G$.}
$$  
The coalgebra structure of $H_{\mathcal{D}}$ is determined by
$$
\Delta(x)\coloneqq x\ot a + 1\ot x\quad\text{and}\quad \Delta(g)\coloneqq g\ot g\quad\text{for all $g\in G$.}
$$
As usual, we let $\epsilon$ and $S$ denote the counit and the antipode of $H_{\mathcal{D}}$, respectively. Note that $\epsilon(g)=1$ and $S(g)=g^{-1}$ for all $g\in G$, and that $\epsilon(x)=0$ and $S(x)=-xa^{-1}=-\chi(a)^{-1}a^{-1}x$. A basis of $H_{\mathcal{D}}$ as $\mathds{k}$-vector space is $\{gx^m : g\in G\text{ and } 0\le m<n\}$. 

\begin{remark}\label{formula de Delta} Let $q\coloneqq \chi(a)$. Since $(1\ot x)(x\ot a)=q(x\ot a)(1\ot x)$, we have
$$
\Delta(gx^m)=\sum_{0\le k\le m} \binom{m}{k}_{\!q} gx^k\ot ga^kx^{m-k}\quad \text{for all $g\in G$ and $m\ge 0$.}
$$
\end{remark}

\begin{remark} The group $G$ is the set of group like elements of $H_{\mathcal{D}}$.
\end{remark}

\begin{remark} In \cite{KR}*{Theorem~1} it was proved that every finite-dimensional pointed Hopf algebra of rank one over $\mathds{k}$ is isomorphic to $H_{\mathcal{D}}$ for some datum $\mathcal{D}$.  
\end{remark}

\section[Hopf \texorpdfstring{$q$}{q}-braces structures on rank one pointed Hopf algebras]{Hopf $\mathbf{q}$-braces structures on rank one pointed Hopf algebras}\label{section 2}

In this section we fix a datum $\mathcal{D}=(G,\chi,a,\alpha)$ and a Hopf $q$-brace structure $\mathcal{H}=(H_{\mathcal{D}},\cdot,\dpu)$ on $H_{\mathcal{D}}$. Note that $\mathcal{H}$ induces by restriction a $q$-brace structure on $G$ (because $\cdot$ and $\dpu$ send group like elements to group like elements). From now on we set
\begin{alignat*}{3}
& x\cdot g=\sum_{\substack{h\in G \\ 0\le m< n}} \lambda^{xg}_{hm} hx^m, &&\qquad g\cdot x=\sum_{\substack{h\in G \\ 0\le m< n}} \lambda^{gx}_{hm} hx^m, &&\qquad x\cdot x=\sum_{\substack{h\in G \\ 0\le m< n}} \lambda_{hm} hx^m,\\
& x\dpu g=\sum_{\substack{h\in G \\ 0\le m< n}} \xi^{xg}_{hm} hx^m, &&\qquad g\dpu x=\sum_{\substack{h\in G \\ 0\le m< n}} \xi^{gx}_{hm} hx^m, &&\qquad x\dpu x =\sum_{\substack{h\in G \\ 0\le m< n}} \xi_{hm} hx^m,
\end{alignat*}
where $\lambda^{xg}_{hm}, \lambda^{gx}_{hm}, \lambda_{hm}, \xi^{xg}_{hm}, \xi^{gx}_{hm}, \xi_{hm}\in \mathds{k}$.

\begin{lemma}\label{x cdot g} The following facts hold:
	
\begin{enumerate}[itemsep=0.7ex, topsep=1.0ex, label={\emph{\arabic*)}}]
		
\item We have $a\cdot g = a\dpu g=a$, for all $g\in G$. 
	
\item For all $g\in G$, we have $x\cdot g= \lambda^{xg}_{10}(1-a) + \lambda^{xg}_{11} x$ and $x\dpu g = \xi^{xg}_{10}(1-a) + \xi^{xg}_{11} x$.

\item For all $g,h\in G$, we have
$$
\qquad\quad \lambda^{x,gh}_{10}= \lambda^{xh}_{10} + \lambda^{xh}_{11} \lambda^{xg}_{10}, \quad \xi^{x,gh}_{10}= \xi^{xh}_{10} + \xi^{xh}_{11} \xi^{xg}_{10}, \quad \lambda^{x,gh}_{11}= \lambda^{xg}_{11}\lambda^{xh}_{11} \quad \text{and}\quad \xi^{x,gh}_{11}= \xi^{xg}_{11}\xi^{xh}_{11}.
$$
\end{enumerate}
\end{lemma}

\begin{proof} We prove the results for $\cdot$ and leave the results for $\dpu$ to the reader. From the identities
\begin{align*}
\sum_{\substack{h\in G \\ 0\le m< n}} \sum_{0\le k\le m} \lambda^{xg}_{hm} \binom{m}{k}_{\!q} hx^k\ot ha^kx^{m-k} &= \Delta(x\cdot g)\\
&= x\cdot g\ot a\cdot g + 1\cdot g \ot x\cdot g\\
&=\sum_{\substack{h\in G \\ 0\le m< n}} \lambda^{xg}_{hm} hx^m \ot a\cdot g + 1\ot \sum_{\substack{h\in G \\ 0\le m< n}} \lambda^{xg}_{hm} hx^m, 
\end{align*}
it follows that $\lambda^{xg}_{hm}=0$, for $m>1$; and also that
\begin{align}
& \sum_{h\in G} \lambda^{xg}_{h0} h\ot h= \sum_{h\in G}\lambda^{xg}_{h0}  h\ot  a\cdot g  + 1\ot \sum_{h\in G} \lambda^{xg}_{h0} h\label{eq1}  
\shortintertext{and}
&\sum_{h\in G} \lambda^{xg}_{h1} (hx\ot ha + h\ot hx) =  \sum_{h\in G} \lambda^{xg}_{h1} hx \ot a\cdot g + 1\ot \sum_{h\in G} \lambda^{xg}_{h1} hx. \label{eq2}
\end{align}
Furthermore, since $\mathcal{H}$ is left regular, $x\cdot g\notin \mathds{k}[G]$. So, necessarily, $\lambda^{xg}_{h1}\ne 0$, for some $h$. But identity~\eqref{eq2} implies that $\lambda^{xg}_{h1}=0$ if $ha\ne a\cdot g$ or $h\ne 1$. Thus, $a\cdot g=a$ and~$\lambda^{xg}_{h1}\ne 0$ if and only if $h=1$. Moreover, by identity~\eqref{eq1} we know that $\lambda^{xg}_{10}= -\lambda^{xg}_{a0}$ and $\lambda^{xg}_{h0}=0$, for $h\notin \{1,a\}$. So, items~1) and~2) are true for $\cdot$. Finally, for all $g,h\in G$, we have
$$
\lambda^{x,gh}_{10}(1-a) + \lambda^{x,gh}_{11} x=(x\cdot h)\cdot g= \bigl(\lambda^{xh}_{10}(1-a) + \lambda^{xh}_{11} x\bigr)\cdot g = \lambda^{xh}_{10}(1-a) + \lambda^{xh}_{11} \lambda^{xg}_{10}(1-a) + \lambda^{xh}_{11} \lambda^{xg}_{11}x,
$$ 
from which the first and third equalities in item~3) follow.
\end{proof}	

\begin{corollary}\label{a^r cdot g} We have $a^r\cdot g = a^r\dpu g=a^r$, for all $g\in G$ and $r\in \mathds{N}_0$.
\end{corollary} 

\begin{proof} For $r=0$ this is trivial and for $r=1$ it is true by Lemma~\ref{x cdot g}(1). Assume it is true for $r$. Then
$$
a^{r+1}\cdot g = (a^r\cdot (g\dpu a))(a\cdot g) = a^r a = a^{r+1}.
$$
The proof for $\dpu$ is similar.
\end{proof}

\begin{lemma}\label{g cdot x y g dpu x} The equalities $g\cdot a = g\dpu a = g$ and $g\cdot x =g\dpu x = 0$ hold, for all $g\in G$.
\end{lemma}

\begin{proof} By Lemma~\ref{x cdot g}, condition~\eqref{intercambio . :} with $u=x$ and $v=g\in G$ says that
\begin{align*}
g\dpu a\ot \lambda_{10}^{xg}(1-a) + g\dpu a\ot \lambda_{11}^{xg} x + g\dpu x\ot 1 & = g\dpu a\ot x\cdot g + g\dpu x\ot 1\cdot g\\
& = g\dpu x\ot a\cdot g + g\dpu 1\ot x\cdot g\\
& = g\dpu x\ot a + g\ot \lambda_{10}^{xg}(1-a) + g\ot \lambda_{11}^{xg} x.
\end{align*}
Since $\lambda_{11}^{xg}\ne 0$, this implies that $g\dpu a = g$ and $g\dpu x = 0$. The same computation, but with $u=g$ and $v=x$ gives $g\cdot a = g$ and $g\cdot x = 0$.
\end{proof}

\begin{lemma}\label{x cdot x y x dpu x} The equalities $x\cdot x= \lambda_{10}(1-a) + \lambda_{11} x$ and $x\dpu x=\xi_{10}(1-a) + \xi_{11} x$ hold.
\end{lemma}

\begin{proof} Since $p\colon H_{\mathcal{D}}\ot H_{\mathcal{D}}^{\cop}\to H_{\mathcal{D}}$ is a coalgebra map and $1\cdot x=0$ and $a\cdot x=0$ (by Lemma~\ref{g cdot x y g dpu x}), we have 
\begin{align*}
\sum_{\substack{h\in G \\ 0\le m< n}} \sum_{0\le k\le m} \lambda_{hm} \binom{m}{k}_{\!q} hx^k\ot ha^kx^{m-k} &= \Delta(x\cdot x)\\
&= x\cdot a\ot a\cdot x + x\cdot x \ot a\cdot 1 + 1\cdot a \ot x\cdot x + 1\cdot x\ot x\cdot 1\\
&= x\cdot x \ot a + 1\ot x\cdot x\\
&=\sum_{\substack{h\in G \\ 0\le m< n}} \lambda_{hm} hx^m \ot a+ 1\ot \sum_{\substack{h\in G \\ 0\le m< n}} \lambda_{hm} hx^m.
\end{align*}	
Hence, $\lambda_{hm}=0$, for $m> 1$;  
\begin{align}
& \sum_{h\in G} \lambda^{x}_{h0} h\ot h= \sum_{h\in G} \lambda_{h0} h\ot a + 1\ot \sum_{h\in G} \lambda^{x}_{h0} h\label{eq7}
\shortintertext{and}
& \sum_{h\in G} \lambda^{x}_{h1} (hx\ot ha + h\ot hx) =\sum_{h\in G} \lambda_{h1} hx\ot a + 1\ot \sum_{h\in G} \lambda_{h1} hx.\label{eq8}
\end{align}
From identity~\eqref{eq7} it follows that $\lambda_{h0}=0$ for all $h\notin \{1,a\}$ and that $ \lambda_{a0}= -\lambda_{10}$; while from identity~\eqref{eq8} it follows that $\lambda_{h1}=0$ if $h\ne 1$. The proof for $\dpu$ is similar.
\end{proof}

\begin{lemma}\label{lambda{x,g dpu a}{10} vs lambda{x,g}{10}} We have $x^r\cdot g = \bigl(\lambda^{xg}_{11}\bigr)^r x^r$ and $x^r\dpu g = \bigl(\xi^{xg}_{11}\bigr)^r x^r$, for all $g\in G$ and $r\in \mathds{N}$.
\end{lemma}
	
\begin{proof} We prove the first equality and leave the second one to the reader. Using~\eqref{condicion q-braza}, Lemma~\ref{x cdot g}(2), and the facts that $a\cdot g = a$, $g\dpu a = g$, $xq=qax$ and $g\dpu x = 0$, we obtain
\begin{align*}
\lambda^{xg}_{10}(1-a)a + \lambda^{xg}_{11}xa &=(x\cdot (g\dpu a))(a\cdot g)\\
&=xa\cdot g\\
&=qax\cdot g\\
&= q\bigl((a\cdot (g\dpu a))(x\cdot g) + (a\cdot (g\dpu x))(1\cdot g)\bigr)\\
%
%
&= qa \bigl(\lambda^{xg}_{10}(1-a) + \lambda^{xg}_{11}x \bigr)\\
&= q\lambda^{xg}_{10}(1-a)a + \lambda^{xg}_{11}xa.
\end{align*}
Since $q$ is a primitive $n$-th root of $1$, this implies that $\lambda^{xg}_{10} = 0$ and shows that $x\cdot g = \lambda^{xg}_{11} x$. In order to finish the proof it suffices to show that $x^r\cdot g=(x\cdot g)^r$, for all $g\in G$ and $r>0$. We proceed by induction in $r$. When $r=1$, this is trivial, and assuming that it is true for $r$, we have
$$
x^{r+1}\cdot g = (x^r \cdot (g\dpu a))(x\cdot g) + (x^r\cdot (g\dpu x))(1\cdot g)= (x^r\cdot g)(x\cdot g) = (x\cdot g)^{r+1},
$$
where we have used Lemmas~\ref{x cdot g}(1) and~\ref{g cdot x y g dpu x}. 
\end{proof}
	
\begin{lemma}\label{calculo de lambda^{xa}_{10}, etc} Assume that $\alpha(a^n-1)\ne 0$. Then $\bigl(\lambda^{xg}_{11}\bigr)^n = \bigl(\xi^{xg}_{11}\bigr)^n=1$, for all $g\in G$.
\end{lemma}	
	
\begin{proof} By Corollary~\ref{a^r cdot g} and Lemma~\ref{lambda{x,g dpu a}{10} vs lambda{x,g}{10}}, we have
$$
\alpha(a^n-1) = \alpha(a^n-1)\cdot g = x^n\cdot g = \bigl(\lambda^{xg}_{11}\bigr)^n x^n = \bigl(\lambda^{xg}_{11}\bigr)^n \alpha(a^n-1).
$$
So, $\bigl(\lambda^{xg}_{11}\bigr)^n = 1$. The proof of the remaining assertion is similar.
\end{proof}

\begin{lemma}\label{lambda x {11}=xi x {11}=0} The equalities $x\cdot x = \lambda_{10}(1-a)$ and $x\dpu x = \xi_{10}(1-a)$ hold. 
\end{lemma}

\begin{proof} An inductive argument using Lemma~\ref{g cdot x y g dpu x} and that $H_{\mathcal{D}}$ is a right $H_{\mathcal{D}}^{\op}$-module via $\cdot$, proves that 
$$
x\cdot x^r=\bigl(\lambda_{11}\bigr)^{r-1}\bigl(\lambda_{10}(1-a) + \lambda_{11} x\bigr)\quad\text{for all $r\ge 0$.}
$$
Thus, 
$$
\bigl(\lambda_{11}\bigr)^{n-1}\bigl(\lambda_{10}(1-a) + \lambda_{11} x\bigr)= x\cdot x^n=x\cdot \alpha(a^n - 1) = 0,	
$$
where the last equality is trivial if $\alpha(a^n - 1) = 0$, and it follows from Lemma~\ref{x cdot g}(3) and Lemma~\ref{calculo de lambda^{xa}_{10}, etc}, otherwise. Therefore, $\lambda_{11}=0$, as desired. The proof of the remaining assertion is similar.
\end{proof}

\begin{proposition}\label{resumen} The following facts hold:
\begin{enumerate}[itemsep=0.7ex, topsep=1.0ex, label={\emph{\arabic*)}}]

\item $x\cdot x = \lambda_{10} (1-a)$ and $x\dpu x = \xi_{10} (1-a)$.

\item $a^r\cdot g = a^r\dpu g = a^r$, for all $g\in G$ and $r\in \mathds{N}_0$. 

\item $g\cdot hx^r = g\dpu hx^r=0$, for all $h,g\in G$ and $r\in \mathds{N}$.

\item $g\cdot a^r = g\dpu a^r=g$, for all $g\in G$ and $r\in \mathds{N}_0$.

\item $hx^r\cdot g = (\lambda^{xg}_{11})^r (h\cdot g) x^r$ and $hx^r\dpu g = (\xi^{xg}_{11})^r (h\dpu g) x^r$, for all $h,g\in G$ and $r\in \mathds{N}$.

\item $hx^r\cdot \alpha g(a^n-1) = 0$ and $hx^r\dpu \alpha g(a^n-1) = 0$, for all $h,g\in G$ and $0\le r<n$.

\end{enumerate}
\end{proposition}

\begin{proof} Item~1) is Lemma~\ref{lambda x {11}=xi x {11}=0} and item~2) is Corollary~\ref{a^r cdot g}. Items~3) and~4) follow from Lemma~\ref{g cdot x y g dpu x} and the fact that $H_{\mathcal{D}}$ is an $H_{\mathcal{D}}^{\op}$-module both via $\cdot$ and via $\dpu$. We next prove the first equality in item~5) and leave the second one to the reader. By the first equality in~\eqref{condicion q-braza} and items~3) and~4), we have
$$
hx^r\cdot g= \sum_{0\le k\le r} \binom{r}{k}_{\!q} (h\cdot (g\dpu a^kx^{r-k}))(x^k\cdot g) = (h\cdot (g \dpu a^r))(x^r\cdot g) = (h\cdot g)(x^r\cdot g) = (\lambda^{xg}_{11})^r (h\cdot g) x^r,
$$
as desired. We finally prove the first equality in item~6) and leave the second one, which is similar, to the reader. If $\alpha (a^n-1) = 0$, this is trivial. Assume that  $\alpha (a^n-1) \ne 0$. Then, by items~4) and~5) and Lemmas~\ref{x cdot g}(3) and~\ref{calculo de lambda^{xa}_{10}, etc}, we have $hx^i\cdot a^n = (\lambda^{xa}_{11})^{ni} (h\cdot a^n) x^i = hx^i$, and so
$$
hx^i\cdot \alpha g(a^n-1) = \alpha (hx^i\cdot (a^n-1))\cdot g = 0, 
$$
as we want.
\end{proof}

\begin{lemma}\label{posibilidades} The following facts hold:
		
\begin{enumerate}[itemsep=0.7ex, topsep=1.0ex, label={\emph{\arabic*)}}]
			
\item $x\cdot x=0$ or $x\dpu g=\chi(g)x$ for all $g\in G$.
			
\item $x\dpu x=0$ or $x\cdot g=\chi(g)x$ for all $g\in G$.
\end{enumerate}
		
\end{lemma}
	
\begin{proof} By Lemmas~\ref{x cdot g} and~\ref{lambda{x,g dpu a}{10} vs lambda{x,g}{10}}, there exist $\mathds{k}^{\times}$-valued characters $g\mapsto \lambda^{xg}_{11}$ and $g\mapsto \xi^{xg}_{11}$, such that  
$$
x\cdot g = \lambda^{xg}_{11}x\quad\text{and}\quad x\dpu g = \xi^{xg}_{11}x. 
$$
Moreover, by Lemma~\ref{g cdot x y g dpu x}, we have $g\cdot x=g\dpu x=0$ and $g\cdot a=g\dpu a=g$, for all $g\in G$. On the other hand, by Lemma~\ref{lambda x {11}=xi x {11}=0}, we have
$$
x\cdot x=\lambda_{10}(1-a) \quad\text{and}\quad x\dpu x=\xi_{10}(1-a).
$$
Using these facts and condition~\eqref{condicion q-braza}, we obtain
$$
gx\cdot x = (g\cdot (x \dpu a))( x \cdot  a) + (g\cdot (1 \dpu a))(x \cdot x) + (g\cdot (x \dpu x))(1 \cdot a) + (g\cdot (1 \dpu x))(1 \cdot x)= \lambda_{10}(1-a)g
$$
and  
$$
xg\cdot x = (x\cdot ( x \dpu g) )(g \cdot a) +  (x\cdot ( 1 \dpu g) )(g \cdot x)= \xi^{xg}_{11}(x\cdot  x)g= \xi^{xg}_{11} \lambda_{10}(1-a)g.
$$
Since $\chi(g)gx = xg$, this implies the first assertion. The second assertion can be proved in a similar way.
\end{proof}

\begin{lemma}\label{restriccion} The equalities $\lambda_{11}^{xa}\xi_{11}^{xa} = 1$ and $\xi_{10}=-\lambda_{10}$ holds. 
\end{lemma}
	
\begin{proof} By Proposition~\ref{resumen} and Condition~\eqref{intercambio . :} with $u=v=x$, we have 
\begin{align*}
a\ot \lambda_{10} (1-a) + \xi_{10} (1-a)\ot a + x\ot x & = a\ot x\cdot x + x\dpu x\ot a + x\ot x\\
& = a\dpu x\ot a\cdot x + a\dpu 1\ot x\cdot x + x\dpu x\ot a\cdot 1 + x\dpu 1\ot x\cdot 1\\ 
& = x\dpu a\ot x\cdot a + x\dpu x\ot 1\cdot a + 1\dpu a\ot x\cdot x + 1\dpu x\ot 1\cdot x\\ 
& = \lambda_{11}^{xa}\xi_{11}^{xa} x\ot x + \xi_{10} (1-a)\ot 1 + 1\ot \lambda_{10} (1-a).
\end{align*}
The results follow immediately from this fact.
\end{proof}

\section[Classification of Hopf \texorpdfstring{$q$}{q}-braces on rank one pointed Hopf algebras]{Classification of Hopf $\mathbf{q}$-braces on rank one pointed Hopf algebras}\label{Section 3}
In this section $\mathcal{D}=(G,\chi,a,\alpha)$ is a datum and $H_{\mathcal{D}}$ is the rank one pointed Hopf algebra associated with $\mathcal{D}$. By~Pro\-position~\ref{resumen}(1) and Lemmas~\ref{posibilidades} and~\ref{restriccion}, in order to classify all the Hopf $q$-braces $\mathcal{H} = (H_{\mathcal{D}},\cdot,\dpu)$, it suffices to consider the cases $x\cdot x = x\dpu x = 0$ and $x\cdot x = - x\dpu x\ne 0$. In the proof of Theorems~\ref{clasificacion 1} and~\ref{n=2} we sometimes will use, without explicit mention, the results of Section~\ref{section 2}, mainly Proposition~\ref{resumen}. We also will use that $l\dpu ha^j = l\dpu h$ and $l\cdot ha^{\jmath} = l\cdot h$, for all $l,h\in G$ and $\jmath\in \mathds{N}$.
	
\begin{remark} Let $(G,\cdot,\dpu)$ be a $q$-brace and let $a\in G$. By condi\-tion~\eqref{condicion q-braza}, if $a\cdot g = a\dpu g = a$ and $g\cdot a = g\dpu a = g$, for all $g\in G$, then $ha\cdot g = (h\cdot g)a$ and $ha\dpu g = (h\dpu g)a$, for all $h,g\in G$. We will use freely this fact.  
\end{remark}

\begin{theorem}\label{clasificacion 1} Let $\mathcal{H}=(H_{\mathcal{D}},\cdot,\dpu)$ be a Hopf $q$-brace structure on $H_{\mathcal{D}}$. Assume that $x\cdot x=x\dpu x=0$.~Then,~$\mathcal{H}$ induces by restriction a $q$-brace structure on $G$ such that $a\cdot g=a\dpu g=a$ and $g\cdot a=g\dpu a=g$, for all $g\in G$. Also, there exist $\mathds{k}^{\times}$-valued characters $\lambda$ and $\xi$, of~$G$, satisfying
\begin{enumerate}[itemsep=0.7ex, topsep=1.0ex, label={\emph{\arabic*)}}]
    		
\item $\lambda(a)=\xi^{-1}(a)$,
    		
\item $\lambda(h(l\dpu h)) = \lambda(l(h\cdot l))$ for all $h,l\in G$,
    		
\item $\lambda(h)\xi(l\cdot h)= \xi(l)\lambda(h\dpu l)$ for all $h,l\in G$,
    		
\item $\xi(l(h\cdot l))=\xi(h(l\dpu h))$ for all $h,l\in G$,
    		
\item $\chi(h)\lambda(l)=\lambda(l\dpu h)\chi(h\cdot l)$ for all $h,l\in G$,
    		
\item $\chi(h)\xi(l)=\xi(l\cdot  h)\chi(h\dpu l)$ for all $h,l\in G$,
    		
\end{enumerate}
such that, for all $h,g\in G$ and $\jmath,\jmath'\in \mathds{N}_0$,  
\begin{equation}\label{cdot y dpu 1}
h x^{\jmath}\cdot g x^{\jmath'}= \begin{cases} \lambda^{\jmath}(g) (h\cdot g) x^{\jmath} &\text{if $\jmath'=0$,}\\0 & \text{otherwise,}\end{cases} \qquad\text{and}\qquad h x^{\jmath}\dpu g x^{\jmath'}= \begin{cases} \xi^{\jmath}(g) (h\dpu g) x^{\jmath} &\text{if $j'=0$,}\\ 0 & \text{otherwise.} \end{cases}
\end{equation}
Conversely, given a $q$-brace structure on $G$ such that $a\cdot g=a\dpu g=a$ and $g\cdot a=g\dpu a=g$, for all $g\in G$, and $\mathds{k}^{\times}$-valued characters $\lambda$ and $\xi$, of $G$, satisfying items~1)--6), the Hopf algebra $H_{\mathcal{D}}$ is a Hopf $q$-brace via the formulas~\eqref{cdot y dpu 1}. 
\end{theorem}
    
\begin{proof} Assume for a while that $\mathcal{H}=(H_{\mathcal{D}},\cdot,\dpu)$ is a Hopf $q$-brace, with $x\cdot x = 0$ and~$x\dpu x = 0$. At the be\-ginning of Section~\ref{section 2}, we saw that $\mathcal{H}$ induces by restriction a $q$-brace structure on $G$. Also, by items~2) and 4) of Proposition~\ref{resumen}, we know that $a\cdot g=a\dpu g=a$ and $g\cdot a=g\dpu a=g$, for all $g\in G$. We claim that 
\begin{equation}\label{eq9}
gx^r\cdot x=0\quad\text{and}\quad gx^r\dpu x = 0\qquad\text{for all $g\in G$ and $r\in \mathds{N}_0$.} 
\end{equation}
We prove the first equality and leave the second one to the reader. Since $gx^r\cdot x = \chi(g)^{-r}x^rg\cdot x$, by~con\-dition~\eqref{condicion q-braza} and items~3) and~5) of Proposition~\ref{resumen}, we have  
$$
gx^r\cdot x =\chi(g)^{-r}(x^r\cdot (x\dpu g))(g\cdot a) + \chi(g)^{-r}(x^r\cdot (1\dpu g))(g\cdot x) =\chi(g)^{-r}\xi_{11}^{xg}(x^r\cdot x)(g\cdot a). 
$$
So, we can suppose that $g=1$. We now proceed by induction on $r$. For $r=0$ this is true by Proposition~\ref{resumen}(3) and, for $r=1$, by hypothesis.~As\-suming that $x^r\cdot x=0$ and using condition~\eqref{condicion q-braza} and the fact that $x\dpu x = 0$, $x\dpu a = \xi_{11}^{xa} x$ and $1\cdot x = 0$, we obtain
$$
x^{r+1}\cdot x = (x^r\cdot (x\dpu a))(x\cdot a) + (x^r\cdot (x\dpu x))(1 \cdot a) + (x^r\cdot (1 \dpu a))(x \cdot x) + (x^r\cdot (1 \dpu x))(1 \cdot x)=0,
$$
and so $x^r\cdot x=0$, for all $r\in \mathds{N}$. Since $H_{\mathcal{D}}$ is a right $H_{\mathcal{D}}^{\op}$-module via $\cdot$ and $\dpu$, it follows from Proposition~\ref{resumen}(5) and condition~\eqref{eq9}, that identities~\eqref{cdot y dpu 1} hold with $\lambda(g)\coloneqq \lambda^{xg}_{11}$ and~$\xi(g)\coloneqq \xi^{xg}_{11}$. Note that, by Lemma~\ref{x cdot g}(3), the maps $\lambda$ and $\xi$ are characters from $G$ to $\mathds{k}^{\times}$. It remains to check that $H_{\mathcal{D}}$ is a Hopf $q$-brace via the maps $\cdot$ and $\dpu$ defined by these formulas if and only if conditions~1)--6) are satisfied.
    	
\smallskip
    	
\noindent \textsc{$\mathcal{H}$ is a $q$-magma coalgebra:}\enspace A direct computation proves that if $\jmath>0$, then
\begin{equation*}
(p\ot p)\Delta_{H_{\mathcal{D}} \ot H^{\op}_{\mathcal{D}}}(hx^{\imath}\ot gx^{\jmath})=0=\Delta_{H_{\mathcal{D}}}p(hx^{\imath}\ot gx^{\jmath}).
\end{equation*}
Moreover, 
\begin{align*}
\Delta_{H_{\mathcal{D}}} p(hx^{\imath}\ot g) &=  \Delta_{H_{\mathcal{D}}}(\lambda^{\imath}(g) (h\cdot g)x^{\imath})\\
%
%
&= \sum_{0\le k\le \imath} \binom{\imath}{k}_{\!q} \lambda^k(g)(h\cdot g)x^k\ot \lambda^{\imath-k}(g)(h\cdot g)a^kx^{\imath-k}\\
&= (p\ot p)\Delta_{H_{\mathcal{D}} \ot H^{\op}_{\mathcal{D}}}(hx^{\imath}\ot g).
\end{align*}
This proves that $p\colon H_{\mathcal{D}}\ot H_{\mathcal{D}}^{\cop}\to H_{\mathcal{D}}$ is a coalgebra morphism. Similarly $d\colon H_{\mathcal{D}}\ot H_{\mathcal{D}}^{\cop}\to H_{\mathcal{D}}$ is also. So, in order to conclude that $\mathcal{H}$ is a $q$-magma coalgebra, it suffices to check identity~\eqref{intercambio . :}. But, 
$$
(gx^{\imath})_{(1)}\dpu (hx^{\jmath})_{(2)}\ot (hx^{\jmath})_{(1)}\cdot (gx^{\imath})_{(2)} =  gx^{\imath} \dpu ha^{\jmath} \ot hx^{\jmath}\cdot  ga^{\imath}= (g\dpu h)\xi^{\imath}(ha^{\jmath})x^{\imath}\ot (h\cdot g)\lambda^{\jmath}(ga^{\imath})x^{\jmath}
$$
and
$$
(gx^{\imath})_{(2)}\dpu (hx^{\jmath})_{(1)}\ot (hx^{\jmath})_{(2)}\cdot (gx^{\imath})_{(1)} = gx^{\imath} \dpu h \ot hx^{\jmath}\cdot g= (g\dpu h)\xi^{\imath}(h)x^{\imath}\ot (h\cdot g)\lambda^{\jmath}(g)x^{\jmath},
$$
which shows that $\mathcal{H}$ is a $q$-magma coalgebra if and only if $\lambda(a)=\xi^{-1}(a)$.
    	
\smallskip
    	
\noindent \textsc{$\mathcal{H}$ is regular:}\enspace By the very definitions of $\overline{G}_{\mathcal{H}}$ and $\overline{G}_{\mathcal{H}^{\op}}$, 
$$
\overline{G}_{\mathcal{H}}(gx^{\imath}\ot hx^{\jmath})= \lambda^{\imath}(h)(g\cdot h)x^{\imath}\ot hx^{\jmath}\quad\text{and}\quad \overline{G}_{\mathcal{H}^{\op}}(gx^{\imath}\ot hx^{\jmath}) =\xi^{\imath}(ha^{\jmath})(g\dpu h)x^{\imath}\ot hx^{\jmath}.
$$ 
From these formulas it follows easily that $\mathcal{H}$ is regular.
    	
\smallskip
    	
\noindent \textsc{$\mathcal{H}$ is $q$-cycle coalgebra:}\enspace If $k>0$ or $\jmath>0$, then
$$
(gx^{\imath}\cdot (hx^{\jmath})_{(1)})\cdot (lx^k\dpu (hx^{\jmath})_{(2)}) = 0\quad\text{and}\quad (gx^{\imath}\cdot (lx^k)_{(2)})\cdot (hx^{\jmath}\cdot (lx^k)_{(1)})= 0.
$$
Moreover, 
$$
(gx^{\imath}\cdot h)\cdot (l\dpu h)= \lambda^{\imath}(h)(g\cdot h)x^{\imath}\cdot (l\dpu h)= \lambda^{\imath}(h(l\dpu h))(g\cdot h)\cdot (l\dpu h)x^{\imath}
$$                           
and
$$
(gx^{\imath}\cdot l)\cdot (h\cdot l) = \lambda^{\imath}(l)(g\cdot l)x^{\imath}\cdot (h\cdot l)= \lambda^{\imath}(l(h\cdot l))(g\cdot l)\cdot (h\cdot l)x^{\imath}. 
$$
So, $\mathcal{H}$ satisfies condition~\eqref{1 de def: q cycle coalgebra} if and only if $\lambda(h(l\dpu h)) = \lambda(l(h\cdot l))$ for all $h,l\in G$. Similar computations show that condition~\eqref{2 de def: q cycle coalgebra} is fulfilled if and only if $\lambda(h)\xi(l\cdot h)= \xi(l)\lambda(h\dpu l)$, for all $h,l\in G$; and that condition~\eqref{3 de def: q cycle coalgebra} is fulfilled if and only if $\xi(h(l\dpu h))= \xi(l(h\cdot l))$ for all $h,l\in G$.

\smallskip
    	
\noindent \textsc{$H_{\mathcal{D}}$ is a right $H_{\mathcal{D}}^{\op}$-module via $p$ and $d$:}\enspace From the formulas \eqref{cdot y dpu 1} it follows that, if $0<\jmath + \jmath'\ne n$, then
$$
lx^{\imath}\cdot hx^{\jmath} gx^{\jmath'} = lx^{\imath} \cdot h \xi^j(g) gx^{\jmath+\jmath'} = 0\quad\text{and}\quad (lx^{\imath}\cdot hx^{\jmath})\cdot gx^{\jmath'} =0,
$$
that if $0<\jmath + \jmath' = n$, then, by Proposition~\ref{resumen}(6), 
$$
lx^{\imath}\cdot hx^{\jmath} gx^{\jmath'} = \xi^j(g) lx^{\imath}\cdot hg\alpha (a^n-1) = 0\quad\text{and}\quad (lx^{\imath}\cdot hx^{\jmath})\cdot gx^{\jmath'} =0,
$$
and that
$$
lx^{\imath}\cdot h g= \lambda^{\imath}(hg) (l\cdot hg) x^{\imath}= \lambda^{\imath}(h)\lambda^{\imath}(g) ((l\cdot g)\cdot h) x^{\imath}= (lx^{\imath}\cdot g)\cdot h.
$$
So, $H_{\mathcal{D}}$ is a right $H_{\mathcal{D}}^{\op}$-module via $p$, and similarly, $H_{\mathcal{D}}$ is a right $H_{\mathcal{D}}^{\op}$-module via $d$.
    	
\smallskip
    	
\noindent \textsc{$\mathcal{H}$ satisfies identities~\eqref{condicion q-braza}:}\enspace If $k>0$, then
$$
(gx^{\imath}hx^{\jmath})\cdot lx^k= 0\quad\text{and}\quad (gx^{\imath} \cdot ((lx^k)_{(1)} \dpu (hx^{\jmath})_{(2)}))((hx^{\jmath})_{(1)}\cdot (lx^k)_{(2)})=0.
$$
Assume that $k=0$. Then
\begin{align*}
(gx^{\imath}\cdot (l \dpu (hx^{\jmath})_{(2)}))((hx^{\jmath})_{(1)}\cdot l) &= (gx^{\imath} \cdot (l \dpu ha^{\jmath}))(hx^{\jmath}\cdot l)\\
&= (gx^{\imath} \cdot (l\dpu h))(hx^{\jmath}\cdot l)\\
&= \lambda^{\imath}(l \dpu h) (g\cdot (l \dpu h)) x^{\imath} \lambda^{\jmath}(l) (h\cdot l) x^{\jmath}\\
&= \lambda^{\imath}(l \dpu h)\lambda^{\jmath}(l)\chi^{\imath}(h\cdot l) (g\cdot (l \dpu h)) (h\cdot l) x^{\imath+\jmath}
\end{align*}
and
\begin{equation*}
gx^{\imath}hx^{\jmath}\cdot l= \chi^{\imath}(h)gh x^{\imath+\jmath}\cdot l=  \chi^{\imath}(h) \lambda^{\imath+\jmath}(l)(gh\cdot l)x^{\imath+\jmath}.
\end{equation*}
Hence, the first identity in~\eqref{condicion q-braza} holds if and only if $\chi(h)\lambda(l)=\lambda(l\dpu h)\chi(h\cdot l)$, for all $h,l\in G$. Similarly, the second identity in~\eqref{condicion q-braza} holds if and only if $\chi(h)\xi(l)=\xi(l\cdot h)\chi(h\dpu l)$, for all $h,l\in G$.
\end{proof}

\begin{remark}\label{ejemplo particular} If $(G,\cdot,\dpu)$ is an skew-brace (i.e. $(l\dpu h)h = (h\cdot l)l$ for all $h,l\in G$), then items~2) and~4) of Theorem~\ref{clasificacion 1} are satisfied. 
\end{remark}

\begin{remark}\label{ejemplo particular'} If $h\cdot g = h\dpu g = h$, for all $h,g\in G$, then items~2)--6) of Theorem~\ref{clasificacion 1} are satisfied.
\end{remark}

\begin{remark}\label{Hopf skew braces caso 1} Let $\mathcal{H}=(H_{\mathcal{D}},\cdot,\dpu)$ be a Hopf $q$-brace structure on $H_{\mathcal{D}}$. Assume that $x\cdot x=x\dpu x=0$. Let $\lambda$ and $\xi$ be as in Theorem~\ref{clasificacion 1}. A direct computation shows that, for each $g,h\in G$ and  $0\le \imath,\jmath<n$,
$$
(gx^{\imath}\dpu (hx^{\jmath})_{(2)})(hx^{\jmath})_{(1)}=\xi^{\imath}(h)\chi^{\imath}(h)(g\dpu h)hx^{\imath+\jmath}\quad\text{and}\quad (hx^{\jmath}\cdot (gx^{\imath})_{(1)})(gx^{\imath})_{(2)}=\lambda^{\jmath}(g)\chi^{\jmath}(g)(h\cdot g)gx^{\imath+\jmath}. 
$$	
From this it follows immediately that $\mathcal{H}$ is a Hopf skew-brace if and only if $\mathcal{H}$ induces a skew-brace structure on $G$ and $\xi=\chi^{-1}=\lambda$. Note that, by item~1) of Theorem~\ref{clasificacion 1}, this implies $\chi(a)=-1$, and so $n=2$.
\end{remark}

\begin{example}\label{ejemplo taft} Let $G\coloneqq \langle w\rangle$ be the cyclic group of order $n>1$ and let $\mathcal{D}\coloneqq (G,\chi,a,0)$ be the datum obtained taking $a\coloneqq w$ and $\chi(w)\coloneqq \varrho$, where $\varrho\in \mathds{k}^{\times}$ is a root of~$1$ of order~$n$. The Hopf algebra $H_{\mathcal{D}}$ is the Taft algebra $T_n$. By Theorem~\ref{clasificacion 1} and Remark~\ref{ejemplo particular}, for each character $\lambda\colon G\to \mathds{k}^{\times}$, the Hopf algebra $T_n$ is a Hopf $q$-brace, via 
\begin{equation*} 
h x^{\jmath}\cdot g x^{\jmath'}= \begin{cases} \lambda^{\jmath}(g) h x^{\jmath} &\text{if $\jmath'=0$,}\\0 & \text{otherwise,}\end{cases} \qquad\text{and}\qquad h x^{\jmath}\dpu g x^{\jmath'}= \begin{cases} \lambda^{-\jmath}(g) h x^{\jmath} &\text{if $j'=0$,}\\ 0 & \text{otherwise,}\end{cases}
\end{equation*}
where $h,g\in G$ and $\jmath,\jmath'\in \mathds{N}_0$. Note that, by Theorem~\ref{clasificacion 1} and Proposition~\ref{resumen}(2), these are all the Hopf $q$-brace structures on $T_n$ satisfying $x\cdot x = x\dpu x = 0$.
\end{example}

\begin{example}\label{ejemplo 1} Let $p\in \mathds{N}$ be a prime number and let $0<\nu<\eta<2\nu$. Consider the cyclic group $G\coloneqq \langle w\rangle$, of order $p^{\eta}$, endowed with the (right) skew-brace structure $(G,\cdot,\dpu)$ given by 
$$
w^i\cdot w^j = w^i\dpu w^j \coloneqq w^{i+p^{\nu}ij}. 
$$
Let $\mathcal{D}\coloneqq (G,\chi,a,0)$ be the datum obtained taking $a\coloneqq w^{p^{\eta-\nu}}$ and $\chi(w)\coloneqq \varrho$, where $\varrho\in \mathds{k}^{\times}$ is a root of~$1$ of order~$p^{\nu}$ (note that $a\cdot g=a\dpu g=a$ and $g\cdot a=g\dpu a=g$, for all $g\in G$). Since the order of $\chi(a)$ is $n\coloneqq p^{2\nu-\eta}$, the dimension of $H_{\mathcal{D}}$ as a $\mathds{k}$-vector space is $p^{\eta}p^{2\nu-\eta} = p^{2\nu}$. If $\lambda,\xi\colon G\to \mathds{k}^{\times}$ are characters satisfying $\lambda(a) = \xi^{-1}(a)$ and $\lambda(w^{p^{\nu}}) = 1$, then conditions~1)--6) of Theorem~\ref{clasificacion 1} are fulfilled, and so $H_{\mathcal{D}}$ is a Hopf $q$-brace via the operations~\eqref{cdot y dpu 1}.
\end{example}

\begin{example}\label{ejemplo 2} Let $0<\nu<\eta\le 2\nu$. Consider the cyclic group $G\coloneqq \langle w\rangle$, of order $2^{\eta}$, endowed with the (right) skew-brace structure $(G,\cdot,\dpu)$ given by 
$$
w^i\cdot w^j = w^i\dpu w^j \coloneqq w^{i+2^{\nu}ij}. 
$$
Let $\mathcal{D}\coloneqq (G,\chi,a,0)$ be the datum obtained taking $a\coloneqq w^{2^{\eta-\nu}}$ and $\chi(w)\coloneqq \varrho$, where $\varrho\in \mathds{k}^{\times}$ is a root of~$1$ of order~$2^{\nu+1}$ (as in the previous example, $a\cdot g=a\dpu g=a$ and $g\cdot a=g\dpu a=g$, for all $g\in G$). Since the order of $\chi(a)$ is $n\coloneqq 2^{2\nu+1-\eta}$, the dimension of $H_{\mathcal{D}}$ as a $\mathds{k}$-vector space is $2^{\eta}2^{2\nu+1-\eta} = 2^{2\nu+1}$. If $\lambda,\xi\colon G\to \mathds{k}^{\times}$ are characters satisfying $\lambda(a) = \xi^{-1}(a)$ and $\lambda(w^{2^{\nu}}) = -1$, then conditions~1)--6) of Theorem~\ref{clasificacion 1} are fulfilled, and so $H_{\mathcal{D}}$ is a Hopf $q$-brace via the operations~\eqref{cdot y dpu 1}.
\end{example}

\begin{theorem}\label{n=2} Let $\mathcal{H}=(H_{\mathcal{D}},\cdot,\dpu)$ be a Hopf $q$-brace. Assume that $x\cdot x\ne 0$ and~$x\dpu x\ne 0$. Then~$n=2$,~$\mathcal{H}$~in\-duces by restriction a $q$-brace structure on $G$ such that $a\cdot g=a\dpu g=a$ and $g\cdot a=g\dpu a=g$, for all $g\in G$,  
	
\begin{enumerate}[itemsep=0.7ex, topsep=1.0ex, label={\emph{\arabic*)}}]
		
\item $\chi^2=\ide$,
		
\item $\chi(h)\chi(l\dpu h)= \chi(l)\chi(h\cdot l)$ for all $h,l\in G$,
				
\end{enumerate}
and there exists $\lambda_{10}\in \mathds{k}^{\times}$ such that
\begin{align}
& hx\cdot g=\chi(g)(h\cdot g)x\quad\text{and}\quad hx\dpu g=\chi(g)(h\dpu g)x,\label{formula hx cdot g y hx dpu g}\\
& h\cdot gx= 0 \quad\text{and}\quad  h\dpu gx=0,\label{h cdot gx y h dpu gx}\\
& hx\cdot gx= \lambda_{10}(h\cdot g)(1-a)\quad\text{and}\quad hx\dpu gx= -\lambda_{10}(h\dpu g)(1-a).\label{hx cdot gx y hx dpu gx}
\end{align}
Conversely, if $(G,\cdot,\dpu)$ is a $q$-brace such that $a\cdot g=a\dpu g=a$ and $g\cdot a=g\dpu a=a$, for all $g\in G$,~and~con\-ditions~1) and~2) are fulfilled, then, for each $\lambda_{10}\in \mathds{k}^{\times}$, the Hopf algebra $H_{\mathcal{D}}$ is a Hopf $q$-brace via the $q$-brace structure of $G$ and~\eqref{formula hx cdot g y hx dpu g}--\eqref{hx cdot gx y hx dpu gx}.
\end{theorem}

\begin{proof} Assume for a while that $\mathcal{H}=(H_{\mathcal{D}},\cdot,\dpu)$ is a Hopf $q$-brace, with $x\cdot x\ne 0$ and~$x\dpu x\ne 0$. We have
$$
\lambda_{10}(1-a)= (x\cdot x)\cdot (a \dpu a)= (x\cdot x)\cdot (a\dpu a) + (x\cdot 1)\cdot (a\dpu x)= (x\cdot a)\cdot (x\cdot a)= \chi^2(a)\lambda_{10}(1-a),
$$
where the first equality holds by items~1) and~4) of Proposition~\ref{resumen}; the second one, by Proposition~\ref{resumen}(3); the third one, by condition~\eqref{1 de def: q cycle coalgebra}; and the last one, by Proposition~\ref{resumen}(1) and Lemma~\ref{posibilidades}. So, \hbox{$\chi(a)=\pm 1$}. But $\chi(a)\ne 1$, and hence $\chi(a)=-1$ and $n=2$. The fact that $\mathcal{H}$ induces by restriction a $q$-brace structure on~$G$ such that $a\cdot g=a\dpu g=a$ and $g\cdot a=g\dpu a=g$, for all $g\in G$, follows as in the proof of Theorem~\ref{clasificacion 1}. We will use freely this fact in the rest of the proof. The identities in~\eqref{formula hx cdot g y hx dpu g} follow easily from Proposition~\ref{resumen}(5) and Lemma~\ref{posibilidades}; while the identities in~\eqref{h cdot gx y h dpu gx} hold by Proposition~\ref{resumen}(3). Since $hx = \chi(h)^{-1} xh$, $h\cdot x = 0$, $h\cdot a = h$ and $x\dpu h = \chi(h)x$, by condition~\eqref{condicion q-braza}, we have
\begin{equation*}
hx\cdot gx = \chi(h)^{-1}(xh\cdot x)\cdot g = \chi(h)^{-1}\bigl((x \cdot (x\dpu h))(h\cdot  a) +(x \cdot (1\dpu  h))(h\cdot x)\bigr)\cdot g = (x \cdot x)h \cdot g.
\end{equation*}
Consequently, by Proposition~\ref{resumen}(1) and the fact that $a$ is central, we have 
\begin{equation*}
hx\cdot gx = \lambda_{10}h(1-a) \cdot g = \lambda_{10}(h\cdot g)(1-a),
\end{equation*}
and so the first identity in~\eqref{hx cdot gx y hx dpu gx} holds. A similar computation shows that 
$$
hx\dpu gx= \xi_{10}(h\dpu g)(1-a),
$$
which, by Lemma~\ref{restriccion}, finishes the proof of~\eqref{hx cdot gx y hx dpu gx}. It remains to check that $H_{\mathcal{D}}$ is a Hopf $q$-brace via the maps $\cdot$ and $\dpu$ defined by these formulas if and only if conditions~1) and~2) are satisfied.
    	
\smallskip
		
\noindent \textsc{$\mathcal{H}$ is a $q$-magma coalgebra:}\enspace the identities
\begin{align*}
&\begin{aligned}
\Delta_{H_{\mathcal{D}}}(hx\cdot g) &= \chi(g)(h\cdot g)x\ot (h\cdot g)a + \chi(g) (h\cdot g)\ot (h\cdot g)x\\
&= hx\cdot g\ot ha\cdot g + h\cdot g\ot hx\cdot g\\
&= (p\ot p)\Delta_{H_{\mathcal{D}} \ot H^{\op}_{\mathcal{D}}}(hx\ot g),
\end{aligned}\\
&\Delta_{H_{\mathcal{D}}}(h\cdot gx)=0= (p\ot p)\Delta_{H_{\mathcal{D}} \ot H^{\op}_{\mathcal{D}}}(h\ot gx)
\shortintertext{and}
& \begin{aligned}
\Delta_{H_{\mathcal{D}}}(hx\cdot gx)&= \lambda_{10}(h\cdot g)\ot (h\cdot g) - \lambda_{10}(h\cdot g)a\ot (h\cdot g)a\\
&= \lambda_{10}(h\cdot g)(1-a)\ot (h\cdot g)a + \lambda_{10} h\cdot g\ot (h\cdot g)(1-a)\\
&= hx\cdot ga\ot ha\cdot gx + hx\cdot gx\ot ha\cdot g + h\cdot ga\ot hx\cdot gx + h\cdot gx\ot hx\cdot g\\
&= (p\ot p)\Delta_{H_{\mathcal{D}} \ot H^{\op}_{\mathcal{D}}}(hx\ot gx),
\end{aligned}
\end{align*}
prove that $p$ is a coalgebra morphism, and a similar computation proves that $d$ is also. So, in order to conclude that $\mathcal{H}$ is a $q$-magma coalgebra, it suffices to check identity~\eqref{intercambio . :}. But the equalities
\begin{align*}
& (hx)_{(1)}\dpu g_{(2)}\ot g_{(1)}\cdot (hx)_{(2)} = hx\dpu g\ot g\cdot ha + h\dpu g\ot g\cdot hx = hx\dpu g\ot g\cdot ha,\\
& (hx)_{(2)}\dpu g_{(1)}\ot g_{(2)}\cdot (hx)_{(1)} =ha\dpu g \ot g\cdot hx + hx\dpu g\ot g\cdot h = hx\dpu g\ot g\cdot ha,\\
& h_{(1)}\dpu (gx)_{(2)}\ot (gx)_{(1)}\cdot h_{(2)} = h\dpu ga\ot gx\cdot h + h\dpu gx\ot g\cdot h = h\dpu g\ot gx\cdot h,\\
&h_{(2)}\dpu (gx)_{(1)}\ot (gx)_{(2)}\cdot h_{(1)} = h\dpu gx\ot ga\cdot h + h\dpu g\ot gx\cdot h = h\dpu g\ot gx\cdot h
\shortintertext{and}
& \begin{aligned}
(hx)_{(2)}\dpu (gx)_{(1)}&\ot (gx)_{(2)}\cdot (hx)_{(1)}=ha\dpu gx\ot ga\cdot hx + ha\dpu g\ot gx \cdot hx + hx\dpu gx\ot ga\cdot h + hx\dpu g\ot gx\cdot h\\
&= \lambda_{10}(h\dpu g)a \ot (g\cdot h)(1-a) - \lambda_{10}(h\dpu g)(1-a) \ot (g\cdot h)a + \chi(g)(h\dpu g)x \ot \chi(h)(g\cdot h)x\\	
&= \lambda_{10}(h\dpu g)a \ot g\cdot h - \lambda_{10}(h\dpu g) \ot (g\cdot h)a + \chi(g)(h\dpu g)x \ot \chi(h)(g\cdot h)x\\
&= \chi(ga)(h\dpu g)x \ot \chi(ha)(g\cdot h)x - \lambda_{10}(h\dpu g)(1-a)\ot g\cdot h + h\dpu g \ot \lambda_{10}(g\cdot h)(1-a)\\
&=hx\dpu ga \ot gx \cdot ha + hx\dpu gx \ot g \cdot ha +  h\dpu ga \ot gx \cdot hx + h\dpu gx \ot g \cdot hx\\
&=(hx)_{(1)}\dpu (gx)_{(2)}  \ot (gx)_{(1)} \cdot (hx)_{(2)},
\end{aligned}	
\end{align*}
prove that~\eqref{intercambio . :} is fulfilled.
	
\smallskip

\noindent \textsc{$\mathcal{H}$ is regular:}\enspace a direct computation proves that the maps $\overline{G}_{\mathcal{H}}$ and $\overline{G}_{\mathcal{H}^{\op}}$ are injective, and hence, bijective.

\smallskip

\noindent \textsc{$\mathcal{H}$ is $q$-cycle coalgebra:}\enspace Let $g,h,l\in G$. We divide the proof in several cases. The equalities
\begin{equation*}
(g\cdot h)\cdot (lx\dpu h)=0= (g\cdot la)\cdot (h\cdot lx) + (g\cdot lx)\cdot (h\cdot 1)
\end{equation*}
prove that condition~\eqref{1 de def: q cycle coalgebra} is always fulfilled for $u=g$, $v=h$ and $w=lx$. Similar~computa\-tions show that conditions~\eqref{2 de def: q cycle coalgebra} and~\eqref{3 de def: q cycle coalgebra} are also always satisfied for $u=g$, $v=h$ and $w=lx$. Arguing in the same way we can check that conditions~\eqref{1 de def: q cycle coalgebra}, \eqref{2 de def: q cycle coalgebra} and~\eqref{3 de def: q cycle coalgebra} are always satisfied for $u=g$, $v=hx$ and $w=l$. The equalities 
\begin{equation*}
(gx\cdot h)\cdot (l\dpu h)=\chi(h)\chi(l\dpu h)\bigl((g\cdot h)\cdot (l\dpu h)\bigr)x\quad\text{and}\quad (gx\cdot l)\cdot (h\cdot l)=\chi(l)\chi(h\cdot l)\bigl((g\cdot l)\cdot (h\cdot l)\bigr)x,
\end{equation*}
prove that condition~\eqref{1 de def: q cycle coalgebra} holds for $u=gx$, $v=h$ and $w=l$ if and only if $\chi(h)\chi(l\dpu h)= \chi(l)\chi(h\cdot l)$. Similar computations show that conditions~\eqref{2 de def: q cycle coalgebra} and~\eqref{3 de def: q cycle coalgebra} hold for $u=gx$, $v=h$ and $w=l$ if and only if~$\chi(h)\chi(l\cdot h)=\chi(l)\chi(h\dpu l)$ and $\chi(h)\chi(l\dpu h)=\chi(l)\chi(h\cdot l)$. The equalities 
\begin{align*}
&(g\cdot hx)\cdot (lx \dpu ha) + (g\cdot h)\cdot (lx \dpu hx)= -\lambda_{10} (g\cdot h)\cdot (l \dpu h)(1-a)=0
\shortintertext{and}
&(g \cdot la) \cdot (hx \cdot lx) +  (g \cdot lx) \cdot (hx \cdot l)= \lambda_{10} (g\cdot la)\cdot (h\cdot l)(1-a)=0,
\end{align*}	
prove that condition~\eqref{1 de def: q cycle coalgebra} is always fulfilled for $u=g$, $v=hx$ and $w=lx$. Similar~computa\-tions show that conditions~\eqref{2 de def: q cycle coalgebra} and~\eqref{3 de def: q cycle coalgebra} are also always satisfied for $u=g$, $v=hx$ and $w=lx$. The equalities 
\begin{align*}
&(gx\cdot h)\cdot (lx\dpu h)= \chi^2(h) (g\cdot h)x\cdot (l\dpu h)x= \lambda_{10}\chi^2(h)\bigl((g\cdot h)\cdot (l\dpu h)\bigr)(1-a),
\shortintertext{and}
&(gx\cdot la)\cdot (h\cdot lx) + (gx\cdot lx)\cdot (h\cdot l)= \lambda_{10}(g\cdot l)(1-a)\cdot (h\cdot l)=\lambda_{10}\bigl((g\cdot l)\cdot (h\cdot l)\bigr) (1-a),
\end{align*}
prove that condition~\eqref{1 de def: q cycle coalgebra} is fulfilled for $u=gx$, $v=h$ and $w=lx$ if and only if~$\chi^2(h)=\ide$. Similar computations show that conditions~\eqref{2 de def: q cycle coalgebra} and~\eqref{3 de def: q cycle coalgebra} are also satisfied for $u=gx$, $v=h$ and $w=lx$ if and only if~$\chi^2(h)=\ide$. Arguing in the same way we can check that conditions~\eqref{1 de def: q cycle coalgebra}, \eqref{2 de def: q cycle coalgebra} and~\eqref{3 de def: q cycle coalgebra} are satisfied for $u=gx$, $v=hx$ and $w=l$, if and only if~$\chi^2(l)=\ide$. The equalities
\begin{align*}
& (gx\cdot hx)\cdot (lx \dpu ha) + (gx \cdot h)\cdot (lx\dpu hx)=-\lambda_{10}\chi(h)(g\cdot h)x\cdot (l\dpu h)(1-a) \\
&\phantom{(gx\cdot hx)\cdot (lx \dpu ha) + (gx \cdot h)\cdot (lx\dpu hx)} = -2\lambda_{10}\chi(h)\chi(l\dpu h)\bigl((g\cdot h)\cdot (l\dpu h)\bigr)x
\shortintertext{and}
& (gx\cdot la)\cdot (hx \cdot lx) + (gx\cdot lx)\cdot (hx\cdot l)=-\lambda_{10}\chi(l)(g\cdot l)x\cdot (h\cdot l)(1-a)\\
&\phantom{(gx\cdot la)\cdot (hx \cdot lx) + (gx\cdot lx)\cdot (hx\cdot l)} = -2\lambda_{10}\chi(l)\chi(h\cdot l)\bigl((g\cdot l)\cdot (h\cdot l)\bigr) x,
\end{align*}
prove that condition~\eqref{1 de def: q cycle coalgebra} is fulfilled for $u=gx$, $v=hx$ and $w=lx$ if and only if $\chi(h)\chi(l\dpu h)=\chi(l)\chi(h\cdot l)$. Similar computations show that conditions~\eqref{2 de def: q cycle coalgebra} and~\eqref{3 de def: q cycle coalgebra} hold for $u=gx$, $v=hx$ and $w=lx$ if and only if $\chi(h)\chi(l\cdot h)=\chi(l)\chi(h\dpu l)$ and $\chi(h)\chi(l\dpu h)=\chi(l)\chi(h\cdot l)$.

\smallskip

\noindent \textsc{$H_{\mathcal{D}}$ is a right $H_{\mathcal{D}}^{\op}$-module via $p$ and $d$:}\enspace  the identities
\begin{align*}
&(lx\cdot g)\cdot h= \chi(h)\chi(g)((l\cdot g)\cdot h)x=\chi(hg)(l\cdot hg)x = lx\cdot hg,\\
&(l\cdot gx)\cdot h= 0= l\cdot hgx,\\
&(l\cdot g)\cdot hx=0=l\cdot hxg,\\
&(lx\cdot gx)\cdot h= \lambda_{10} (l\cdot g)(1-a)\cdot h= \lambda_{10} ((l\cdot g)\cdot h)(1-a)= \lambda_{10} (l\cdot hg)(1-a)=lx\cdot hgx,\\
&(lx\cdot g)\cdot hx=\chi(g)(l\cdot g)x\cdot hx=\chi(g)\lambda_{10}((l\cdot g)\cdot h)(1-a)=\chi(g)\lambda_{10}(l\cdot hg)(1-a)=lx \cdot hxg\\
&(l\cdot gx)\cdot hx = 0 = \chi(g)\alpha (l\cdot (1-a))\cdot hg = \chi(g)\alpha l\cdot hg (1-a) = l\cdot hxgx 
\shortintertext{and}
&(lx\cdot gx)\cdot hx= 0= \alpha\chi(g) lx\cdot hg (a^2-1)= lx\cdot hxgx,
\end{align*}	
prove that $H_{\mathcal{D}}$ is a right $H_{\mathcal{D}}^{\op}$-module via $p$. Similarly, $H_{\mathcal{D}}$ is a right $H_{\mathcal{D}}^{\op}$-module via $d$.

\smallskip

\noindent \textsc{$\mathcal{H}$ satisfies the identities~\eqref{condicion q-braza}:}\enspace We begin with the first identity. Let $g,h,l\in G$. The identities
\begin{align*}
&gh\cdot lx =0=(g\cdot (lx \dpu h))(h\cdot la) + (g\cdot (l \dpu h))(h\cdot lx),\\
&ghx\cdot l= \chi(l)(gh\cdot l)x= \chi(l)(g\cdot (l\dpu h))(h\cdot l)x= (g\cdot (l\dpu ha))(hx \cdot l) + (g\cdot (l\dpu hx))(h \cdot l),\\
&gxh\cdot l= \chi(h)\chi(l)(gh\cdot l)x= \chi(h)\chi(l)(g\cdot (l\dpu h))(h\cdot l),\\
&(gx\cdot (l\dpu h))(h\cdot l)= \chi(l\dpu h)\chi(h\cdot l)(g\cdot (l\dpu h))(h\cdot l)x,\\
& \begin{aligned}
ghx\cdot lx &= \lambda_{10}(gh\cdot l)(1-a)\\
&=(g\cdot (l \dpu ha))(hx \cdot lx)\\
&=(g\cdot (lx \dpu ha))(hx \cdot la) + (g\cdot (lx \dpu hx))(h \cdot la) + (g\cdot (l \dpu ha))(hx \cdot lx) + (g\cdot (l \dpu hx))(h \cdot lx),
\end{aligned}\\
&\begin{aligned}
gxh\cdot lx &= \chi(h)\lambda_{10}(gh\cdot l)(1-a)\\
&=\chi(h)\lambda_{10}(g\cdot (l\dpu h))(h\cdot l)(1-a)\\
&=\chi(h)(gx\cdot (l\dpu h)x)(h\cdot l)\\
&=(gx\cdot (lx \dpu h))(h \cdot la) + (gx\cdot (l \dpu h))(h \cdot lx), 
\end{aligned}\\
&gxhx\cdot l=\chi(h)\alpha gh(a^2-1)\cdot l= \chi(h)\alpha (gh\cdot l)(a^2-1)= \chi(h)\alpha (g \cdot (l \dpu h))(h\cdot l)(a^2-1),\\
&(gx \cdot (l \dpu ha))(hx\cdot l) + (gx \cdot (l \dpu hx))(h\cdot l)= \chi(l)\chi(l\dpu h)\chi(h\cdot l)\alpha (g\cdot (l \dpu h))(h\cdot l)(a^2-1),\\
& gxhx\cdot lx=0
\shortintertext{and}
&\begin{aligned}
(gx \cdot &(lx \dpu ha))(hx \cdot la)   + (gx \cdot (lx \dpu hx))(h \cdot la) + (gx \cdot (l \dpu ha))(hx \cdot lx) + (gx \cdot (l \dpu hx))(h \cdot lx)\\
&=\chi(h)\chi(l)(gx\cdot (l\dpu h)x)(h\cdot l)x - \lambda_{10}(gx\cdot (l\dpu h)(1-a))(h\cdot l) + \lambda_{10}(gx \cdot (l \dpu h))(h \cdot l)(1-a)\\
&=\lambda_{10}\chi(h)\chi(l)(g\cdot (l\dpu h))(h\cdot l)x - \lambda_{10}\chi(h)\chi(l)(g\cdot (l\dpu h))(h\cdot l)ax\\
&\phantom{=\, } - 2\lambda_{10}\chi(l\dpu h)\chi(h\cdot l)(g\cdot (l\dpu h))(h\cdot l)x + \lambda_{10}\chi(l\dpu h)\chi(h\cdot l)(g\cdot (l\dpu h))(h\cdot l)x\\
&\phantom{= \,} + \lambda_{10}\chi(l\dpu h)\chi(h\cdot l)(g\cdot (l\dpu h))(h\cdot l)ax,
\end{aligned}
\end{align*} 
prove that $\mathcal{H}$ satisfies the first identity in~\eqref{condicion q-braza} if and only if, for all $h,l\in G$, we have
$$
\chi(h)\chi(l)=\chi(l\dpu h)\chi(h\cdot l)\quad\text{and}\quad \alpha\chi(h)(a^2-1)=\alpha\chi(l)\chi(l\dpu h)\chi(h\cdot l)(a^2-1).
$$
A similar computation shows that $\mathcal{H}$ satisfies the second identity in~\eqref{condicion q-braza} if and only if, for all $h,l\in G$, 
$$
\chi(h)\chi(l)=\chi(l\cdot h)\chi(h\dpu l)\quad\text{and}\quad \alpha\chi(h)(a^2-1)=\alpha\chi(l)\chi(l\cdot h)\chi(h\dpu l)(a^2-1).
$$
The proof follows easily from these facts.
\end{proof}

\begin{remark}\label{ejemplo particular2} If $(G,\cdot,\dpu)$ is a skew-brace, then item~2) of Theorem~\ref{n=2} is satisfied. 
\end{remark}

\begin{example}\label{ejemplo sweedler} Let $T_2$ be as in Example~\ref{ejemplo taft}. By Theorem~\ref{n=2} and Remark~\ref{ejemplo particular2}, for each $\lambda_{10}\in \mathds{k}^{\times}$, the Hopf algebra $T_2$ is a Hopf $q$-brace, via
$$
hx\cdot g= hx\dpu g=\chi(g)hx,\quad h\cdot gx=h\dpu gx=0\quad\text{and}\quad hx\cdot gx=- hx\dpu gx =\lambda_{10} h (1-w),
$$
for each $h,g\in G = \{1,w\}$. Note that, by Theorem~\ref{n=2} and Proposition~\ref{resumen}(2), these are all the Hopf $q$-brace structures on $T_2$, satisfying $x\cdot x\ne 0$ and $x\dpu x\ne 0$. Moreover, when $n>2$, there are no Hopf $q$-brace structures on $T_n$, satisfying $x\cdot x\ne 0$ and $x\dpu x\ne 0$.
\end{example}

\begin{example}\label{ejemplo 3} Let $0<\nu<\eta\le 2\nu$ and $(G,\cdot,\dpu)$ be as in Example~\ref{ejemplo 2}. Let $\mathcal{D}\coloneqq (G,\chi,a,\alpha)$ be the datum obtained taking $a\coloneqq w^{2^{\eta-\nu}+1}$ and $\chi(w)\coloneqq -1$ and $\alpha\in \mathds{k}$, (as in Example~\ref{ejemplo 2}, we have $a\cdot g=a\dpu g=a$ and $g\cdot a=g\dpu a=g$, for all $g\in G$). Since $\chi(a)=-1$, the dimension of $H_{\mathcal{D}}$ as a $\mathds{k}$-vector space is $2^{\eta}2 = 2^{\eta+1}$. Since con\-ditions~1) and~2) of Theorem~\ref{n=2} are fulfilled, $H_{\mathcal{D}}$ is a Hopf $q$-brace via~\eqref{formula hx cdot g y hx dpu g}--\eqref{hx cdot gx y hx dpu gx}, for each $\lambda_{10}\in \mathds{k}^{\times}$.
\end{example}

\begin{remark}\label{Hopf skew braces caso 2} Let $\mathcal{H}=(H_{\mathcal{D}}, \cdot,\dpu)$ be a Hopf $q$-brace structure on $H_{\mathcal{D}}$. Assume that $x\cdot x\ne 0$ and $x\dpu x\ne 0$ and  let $\lambda_{10}$ be as in Theorem~\ref{n=2}. A direct computation shows that
\begin{align*}
& (gx\dpu h_{(2)})h_{(1)} = (gx\dpu h)h = \chi(h) (g\dpu h) x h = (g\dpu h) hx,\\
& (h\cdot (gx)_{(1)})(gx)_{(2)} = (h\cdot gx)ga + (h\cdot g)gx = (h\cdot g)gx,\\
& (g\dpu (hx)_{(2)})(hx)_{(1)} = (g\dpu ha)hx + (g\dpu hx)h = (g\dpu h) hx,\\
& (hx\cdot g_{(1)})g_{(2)} = (hx\cdot g)g = \chi(g) (h\cdot g)xg = (h\cdot g)gx,\\
& (gx\dpu (hx)_{(2)})(hx)_{(1)} = (gx\dpu ha)hx+ (gx\dpu hx)h = -\alpha (g\dpu h)h (a^2-1) - \lambda_{10} (g\dpu h)h(1-a)\\
\shortintertext{and}
& (hx\cdot (gx)_{(1)})(gx)_{(2)} =(hx\cdot gx)ga + (hx\cdot g)gx = \lambda_{10}(h\cdot g)g(a-a^2) + \alpha (h\cdot g)g(a^2-1).
\end{align*}
Hence $\mathcal{H}$ is a Hopf skew-brace if and only if $G$ is a skew-brace and $2\alpha (a^2-1) = \lambda_{10} (a^2-1)$.
\end{remark}

Let $\mathcal{D}=(G,\chi,a,\alpha)$ and $\mathcal{D'}=(G',\chi',a',\alpha')$ be two data. By~\cite{KR}*{Theorem~1} we know that $H_{\mathcal{D}}$ and~$H_{\mathcal{D'}}$ are isomorphic Hopf algebras if and only if there is an isomorphism of groups, $f\colon G\to G'$, such that $f(a) = a'$, $\chi = \chi'\xcirc f$ and $\beta \alpha'({a'}^n-1) = \alpha({a'}^n-1)$, for some $\beta\in \mathds{k}^{\times}$, where $n$ is the order of $\chi(a)$. Moreover, from the proof it follows that each isomorphism $F\colon H_{\mathcal{D}}\to H_{\mathcal{D'}}$ is given by $F(gx^m) = \alpha_0^m f(g)x^m$, where $f$ is as above and $\alpha_0^n = \beta$.

\begin{remark} If $(H_{\mathcal{D}},\cdot,\dpu)$ and $(H_{\mathcal{D}'},\cdot,\dpu)$ are two isomorphic Hopf $q$-braces, then $x\cdot x = 0$ in both, or $x\cdot x \ne 0$ in both.
\end{remark}

\begin{remark} Let $(H_{\mathcal{D}},\cdot,\dpu)$ and $(H_{\mathcal{D}'},\cdot,\dpu)$ be Hopf $q$-braces that satisfy the conditions of Theorem~\ref{clasificacion 1}. Let~$\lambda$ and $\xi$ be the $\mathds{k}^{\times}$-valued characters of~$G$, associated with $(H_{\mathcal{D}},\cdot,\dpu)$ and let $\lambda'$ and $\xi'$ be the $\mathds{k}^{\times}$-valued characters of~$G'$, associated with $(H_{\mathcal{D'}},\cdot,\dpu)$. A direct computation shows that an isomorphism $F\colon H_{\mathcal{D}}\to H_{\mathcal{D'}}$, of Hopf algebras, is an isomorphism of Hopf $q$-braces if and only if $f\colon (G,\cdot,\dpu)\to (G',\cdot,\dpu)$ is an isomorphism, $\lambda = \lambda'\xcirc f$ and $\xi = \xi' \xcirc f$.
\end{remark}

\begin{remark} Let $(H_{\mathcal{D}},\cdot,\dpu)$ and $(H_{\mathcal{D}'},\cdot,\dpu)$ be Hopf $q$-braces that satisfy the conditions of Theorem~\ref{n=2}. Let $\lambda_{10}\in \mathds{k}^{\times}$ be the scalar associated with $(H_{\mathcal{D}},\cdot,\dpu)$ and let $\lambda'_{10}\in \mathds{k}^{\times}$ be the scalar associated with $(H_{\mathcal{D'}},\cdot,\dpu)$. A direct computation shows that an isomorphism $F\colon H_{\mathcal{D}}\to H_{\mathcal{D'}}$, of Hopf algebras, is an isomorphism of Hopf $q$-braces if and only if $f\colon (G,\cdot,\dpu)\to (G',\cdot,\dpu)$ is an isomorphism and $\frac{\lambda_{10}}{\lambda'_{10}} \alpha'({a'}^2-1) = \alpha({a'}^2-1)$.
\end{remark}

\section{The Socle}\label{section 4}

Let $\mathcal{H}=(H,\cdot,\dpu)$ be a Hopf $q$-brace with bijective antipode. Recall from~\cite{GGV}*{Definition~7.11} that the socle of $\mathcal{H}$ is the set
\begin{equation*}
\Soc(\mathcal{H})\coloneqq \{u\in H: u_{(1)}\ot v\cdot u_{(2)}\ot u_{(3)}=u_{(1)}\ot v\dpu u_{(2)}\ot u_{(3)}=u_{(1)}\ot v\ot u_{(2)}\ \forall u\in H\}.
\end{equation*}
In this section we compute the socle of the Hopf $q$-braces constructed in Theorems~\ref{clasificacion 1} and~\ref{n=2}. We treat separately the cases considered in each one of these results. Let $u\coloneqq hx^j$ and $v\coloneqq kx^i$. A direct computation using Remark~\ref{formula de Delta} shows that
\begin{equation}\label{calculo de Delta2}
\Delta^2(u) = \sum_{0\le l_1\le l_2\le j} \binom{j}{l_2}_{\!q} \binom{l_2}{l_1}_{\!q} hx^{l_1}\ot ha^{l_1}x^{l_2-l_1} \ot ha^{l_2}x^{j-l_2}.
\end{equation}

\subsection[Case \texorpdfstring{$x\cdot x = x\dpu x = 0$}{x.x = x:x=0}]{Case $\pmb{x\cdot x = x\dpu x = 0}$} 

Let $u\coloneqq hx^j$ and $v\coloneqq kx^i$. Using~\eqref{cdot y dpu 1}  and~\eqref{calculo de Delta2}, we obtain that $u_{(1)}\ot v\cdot u_{(2)}\ot u_{(3)} = u_{(1)}\ot v\ot u_{(2)}$ if and only if
\begin{equation*}
\sum_{0\le l\le j} \binom{j}{l}_{\!q} hx^l\ot \lambda^i(ha^l)(k\cdot h) x^i\ot ha^lx^{j-l} = \sum_{0\le l\le j} \binom{j}{l}_{\!q} hx^l\ot k x^i\ot ha^lx^{j-l}.
\end{equation*}
Similarly, $u_{(1)}\ot v\dpu u_{(2)}\ot u_{(3)} = u_{(1)}\ot v\ot u_{(2)}$ if and only if 
\begin{equation*}
\sum_{0\le l\le j} \binom{j}{l}_{\!q} hx^l\ot \xi^i(ha^l)(k\dpu h) x^i\ot ha^lx^{j-l} = \sum_{0\le l\le j} \binom{j}{l}_{\!q} hx^l\ot k x^i\ot ha^lx^{j-l}.
\end{equation*}
Using these facts it is easy to see that $\sum_{j=0}^{n-1}\sum_{h\in G} \mu_{j,h} h x^j\in \Soc(\mathcal{H})$ if and only if $h x^j\in \Soc(\mathcal{H})$ for all $j,h$ such that $\mu_{j,h}\ne 0$. Next we search conditions for $hx^j\in \Soc(\mathcal{H})$. By definition
\begin{equation*}
hx^j \in \Soc(\mathcal{H}) \Leftrightarrow \lambda^i(ha^l)(k\cdot h) = \xi^i(ha^l)(k\dpu h) = k\quad\text{for all $0\le i<n$, $0\le l\le j$ and $k\in G$.} 
\end{equation*}
Taking $i=0$, we get that $k\cdot h = k\dpu h = k$, for all $k\in G$. In other words $h\in \Soc(G,\cdot,\dpu)$. So 
\begin{equation*}
hx^j\in\Soc(\mathcal{H})\Leftrightarrow h\in \Soc(G,\cdot,\dpu)\text{ and } \lambda^i(ha^l)= \xi^i(ha^l)= 1\quad\text{for all $0\le i<n$ and $0\le l\le j$.} 
\end{equation*}
Since $\lambda$ and $\chi$ are characters and $\xi(a) = \lambda^{-1}(a)$, we conclude that
\begin{equation}\label{calculo del socalo 1}
hx^j \in \Soc(\mathcal{H}) \Leftrightarrow h\in \Soc(G,\cdot,\dpu)\text{ and } \begin{cases} \lambda(h) = \xi(h) = 1 & \text{if $j = 0$,}\\ \lambda(h) = \xi(h) = \lambda(a) = 1 & \text{if $j>0$.}\end{cases}
\end{equation}
Let $L\coloneqq\{h\in \Soc(G,\cdot,\dpu):\lambda(h) = \xi(h) = 1\}$. From~\eqref{calculo del socalo 1} we obtain that
\begin{equation*}
\Soc(\mathcal{H}) = \begin{cases} \mathds{k}[L]\oplus \mathds{k}[L]x\oplus\cdots\oplus \mathds{k}[L]x^{n-1} &\text{if $\lambda(a) = 1$,}\\ \mathds{k}[L] & \text{if $\lambda(a)\ne 1$.} \end{cases}
\end{equation*}
Consequently, 
\begin{equation*}
\frac{\mathcal{H}}{\Soc(\mathcal{H})^{+}\mathcal{H}} = \begin{cases} \mathds{k}[G/L] &\text{if $\lambda(a) = 1$,}\\ \mathds{k}[G/L]\oplus \mathds{k}[G/L]x\oplus\cdots\oplus \mathds{k}[G/L]x^{n-1} & \text{if $\lambda(a)\ne 1$.} \end{cases}
\end{equation*}
A direct computation shows that the hypotheses in~\cite{GGV}*{Proposition~7.18} are satisfied if and only if $\lambda(a) = 1$. Thus, in this case, $(G/L,\cdot,\dpu)$ is a skew-brace. Note that $L\subseteq \Soc(G)$.

\subsection[Case \texorpdfstring{$x\cdot x \ne 0$}{x.x not equal 0} and \texorpdfstring{$x\dpu x \ne 0$}{x:x not equal 0}]{Case $\pmb{x\cdot x\ne 0}$ and $\pmb{x\dpu x\ne 0}$} 

To begin with we will compute the condition
\begin{equation}\label{condiciones de socalo}
u_{(1)}\ot v \ot u_{(2)} = u_{(1)}\ot v\cdot u_{(2)}\ot u_{(3)} = u_{(1)}\ot v\dpu u_{(2)}\ot u_{(3)}
\end{equation}
in several cases. Using that $\mathcal{H}$ induces by restriction a $q$-brace structure on $G$, and formulas~\eqref{formula hx cdot g y hx dpu g}, \eqref{h cdot gx y h dpu gx}, \eqref{hx cdot gx y hx dpu gx} and~\eqref{calculo de Delta2}, we obtain that

\begin{enumerate}
\item If $u\coloneqq h\in G$ and $v\coloneqq k\in G$, then~\eqref{condiciones de socalo} becomes
\begin{equation*}
h\ot k \ot h = h\ot k\cdot h\ot h = h\ot k\dpu h\ot h.
\end{equation*}

\item If $u\coloneqq h\in G$ and $v\coloneqq kx\in Gx$, then~\eqref{condiciones de socalo} becomes
\begin{equation*}
h\ot kx\ot h= h\ot\chi(h)(k\cdot h)x\ot h = h\ot \chi(h)(k\dpu h)x\ot h.
\end{equation*}

\item If $u\coloneqq hx\in Gx$ and $v\coloneqq k\in G$, then~\eqref{condiciones de socalo} becomes
\begin{align*}
hx\ot k\ot ha + h\ot k\ot hx & = hx\ot k\cdot ha\ot ha + h\ot k\cdot h\ot hx\\ 
& = hx\ot k\dpu ha\ot ha + h\ot k\dpu h\ot hx.
\end{align*}

\item If $u\coloneqq hx\in Gx$ and $v\coloneqq kx\in Gx$, then~\eqref{condiciones de socalo} becomes
\begin{align*}
& hx\ot kx\ot ha + h\ot kx\ot hx\\ 
&= hx\ot \chi(ha)(k\cdot ha)x\ot ha + h\ot \lambda_{10} (k\cdot h)(1-a)\ot ha + h\ot \chi(h)(k\cdot h)x\ot hx\\ 
&= hx\ot \chi(ha)(k\dpu ha)x\ot ha + h\ot \lambda_{10} (k\dpu h)(1-a)\ot ha + h\ot \chi(h)(k\dpu h)x\ot hx.
\end{align*}

\end{enumerate}
From these facts it follows that $\sum_{h\in G} \mu_h h + \sum_{h\in G} \mu'_h hx\in \Soc(\mathcal{H})$ if and only if 

\begin{itemize}[itemsep=0.7ex, topsep=1.0ex]

\item[-] $\mu'_h = 0$, for all $h\in G$, 

\item[-] $h\in \Soc(G,\cdot,\dpu)$ and $\chi(h) = 1$, for all $h$ such that $\mu_h\ne 0$. 

\end{itemize}
In other words $\Soc(\mathcal{H}) = \mathds{k}[L]$, where $L\coloneqq \{h\in \Soc(G,\cdot,\dpu):\chi(h) = 1\}$. Consequently, 
\begin{equation*}
\frac{\mathcal{H}}{\Soc(\mathcal{H})^{+}\mathcal{H}} = \mathds{k}[G/L] \oplus \mathds{k}[G/L]x.
\end{equation*}

\section{Weak braiding structures on rank one Hopf algebras}\label{section 5}

In this section we compute the weak braiding operators associated with the Hopf $q$ braces $\mathcal{H}=(H_{\mathcal{D}},\cdot,\dpu)$,~ob\-tained in~Theo\-rems~\ref{clasificacion 1} and~\ref{n=2}, according to \cite{GGV}*{Corollary~4.11 and Theorem~5.15}.  We will use freely the notations in \cite{GGV}*{Remark~1.17}. Let $s\coloneqq s_{\mathcal{H}}$, so that $s(h \ot l) = {}^{h_{(1)}} l_{(1)}\ot {h_{(2)}}^{l_{(2)}}$. Note that $h^g,{}^h\!g\in G$, for all $h,g\in G$ (because the maps $h\ot l\mapsto h^l$ and $h\ot l\mapsto {}^hl$ send group like elements to group like elements). By items~2) and 4) of
Proposition~\ref{resumen}, we have 
\begin{equation}\label{ecua1}
(a^r)^g = a^r,\quad g^{a^r} = g\quad\text{and}\quad {}^a\hs g = g\dpu a^g = g\qquad \text{for all $g\in G$ and $r\in \mathds{N}_0$.} 
\end{equation}
By item~4) above \cite{GGV}*{Theorem~5.4}, we also have
\begin{equation}\label{ecua2}
(ha)^g = (h^{{}^a\hs g})a^g = (h^g) a\quad \text{for all $h,g\in G$.}
\end{equation}

\begin{proposition} Assume that we are under the hypotheses of Theorem~\ref{clasificacion 1} and let $\mathcal{H}$ be as in that theorem. We have
\begin{equation}\label{pepe}
(hx^j)^{gx^{j'}} = \begin{cases} \lambda^{-j}(g)h^g x^j &\text{if $j'=0$,}\\ 0 &\text{otherwise} \end{cases}\qquad\text{and}\qquad {}^{gx^{j'}}\! (hx^j) = \begin{cases} \xi^j(g^h){}^g\hs h x^j &\text{if $j'=0$,}\\ 0 &\text{otherwise.} \end{cases}
\end{equation}
\end{proposition}

\begin{proof} In order to prove the first equality it suffices to check that identities~(1.2) of \cite{GGV} hold, but this~fol\-lows by a direct computation. The second equality follows directly using Remark~\ref{formula de Delta}, the second equality in~\eqref{cdot y dpu 1} and that ${}^xy\coloneqq y_{(2)}\dpu x^{y_{(1)}}$, for all $x,y\in H_{\mathcal{D}}$ (see \cite{GGV}*{Remark~1.17}).
\end{proof}

\begin{remark} Let $\mathcal{H}$ be as in Theorem~\ref{clasificacion 1} and let $(H_{\mathcal{D}},s)$ be the weak braiding operator associated with~$\mathcal{H}$. Using Remark~\ref{formula de Delta}, the previous proposition and the second equality in~\eqref{ecua1}, we obtain that
$$
s(hx^j\ot gx^{j'}) = {}^h (gx^{j'}) \ot (hx^j)^{ga^{j'}} = \xi^{j'}(h^g){}^h\hs gx^{j'}\ot\lambda^{-j}(ga^{j'})h^gx^j.
$$
\end{remark}

\begin{proposition} Assume that we are under the hypotheses of Theorem~\ref{n=2} and let $\mathcal{H}$ be as in that theorem. We have
\begin{alignat*}{3}
&(hx)^g = \chi^{-1}(g) h^g x,&\qquad h^{gx} = 0,&\qquad (hx)^{gx} = \lambda_{10} \chi^{-1}(g) h^g(a-1),\\
& {}^g\hs (hx) = \chi(g^h){}^g\hs h x,&\qquad {}^{gx}h=0, &\qquad {}^{gx}(hx) = \lambda_{10} \chi^{-1}(h){}^g\hs h(a-1).
\end{alignat*}
\end{proposition}

\begin{proof} In order to prove the first three equalities it suffices to check that identities~(1.2) of \cite{GGV} hold. For the first two identities this is very easy. We next consider the third one. By Theorem~\ref{n=2} and identities~\eqref{ecua1} and~\eqref{ecua2}, we have 
\begin{align*}
(hx\cdot (gx)_{(1)})^{(gx)_{(2)}} & = (hx\cdot g)^{gx} + (hx\cdot gx)^{ga}\\
& = \chi(g)\bigl((h\cdot g)x\bigr)^{gx} + \lambda_{10}\bigl((h\cdot g)(1-a)\bigr)^{ga}\\
& = \chi(g)\bigl((h\cdot g)x\bigr)^{gx} + \lambda_{10}(h\cdot g)^{ga}(1-a)\\
& = \lambda_{10}h(a-1) + \lambda_{10}h(1-a)\\
& = 0,
\end{align*}
which ends the proof of the third identity. The last three equalities follow by a direct computation.
\end{proof}

\begin{remark} Let $\mathcal{H}$ be as in Theorem~\ref{n=2} and let $(H_{\mathcal{D}},s)$ be the weak braiding operator associated with~$\mathcal{H}$. Using Remark~\ref{formula de Delta}, the previous proposition, the second equalities in~\eqref{formula hx cdot g y hx dpu g}, \eqref{h cdot gx y h dpu gx} and~\eqref{hx cdot gx y hx dpu gx} and equalities~\eqref{ecua1} and~\eqref{ecua2}, we obtain that 
\begin{align*}
& s(h\ot gx) = \chi(h^g) {}^h\hs g x\ot h^g,\\
& s(hx\ot g) = \chi^{-1}(g) {}^h\hs g\ot (h^g)x\\
\shortintertext{and}
& s(hx\ot gx) = \lambda_{10} \chi^{-1}(g) {}^h\hs g\ot (h^g)(a-1) + \chi(h^g)\chi^{-1}(ga) {}^h\hs gx\ot (h^g) x + \lambda_{10} \chi^{-1}(g){}^h\hs g(a-1)\ot (h^g)a.
\end{align*}
\end{remark}

\end{document}